\documentclass[11pt]{amsart}
\usepackage{amsmath}
\usepackage{tikz}
\usetikzlibrary{matrix,arrows}
\usepackage{amsfonts,amssymb,amsthm, txfonts, pxfonts,amscd}
\usepackage[stretch=10]{microtype}
\usepackage{hyperref}
\def\struckint{\mathop{%
\def\mathpalette##1##2{\mathchoice{##1\displaystyle##2}%
 {##1\textstyle##2}{##1\scriptstyle##2}{##1\scriptscriptstyle##2}}%
\mathpalette
{\vbox\bgroup\baselineskip0pt\lineskiplimit-1000pt\lineskip-1000pt
\halign\bgroup\hfill$}
{##$\hfill\cr{\intop}\cr\diagup\cr\egroup\egroup}%
}\limits}

\usepackage{mathrsfs}
\def\id{\mathop{\mathbf{Id}_{\mathbb{N}}}\nolimits}
\def\idn{\mathop{\mathbf{Id}_{[n]}}\nolimits}
\def\idtwo{\mathop{\mathbf{Id}_{[2]}}\nolimits}

\def\wireN{\mathop{\mathcal{W}_{\mathbb{N}}}\nolimits}
\def\wirem{\mathop{\mathcal{W}_{[m]}}\nolimits}
\def\wiren{\mathop{\mathcal{W}_{[n]}}\nolimits}

\def\sup{\mathop{\rm sup}\nolimits}

\def\symmetricN{\mathop{\mathfrak{S}_{\mathbb{N}}}\nolimits}

\def\Gbf{\mathop{\mathbf{\Gamma}_{}}\nolimits}

\def\stoch{\mathop{\mathcal{S}_2}}

\def\graphsn{\mathop{\mathcal{G}_{[n]}}\nolimits}
\def\graphsm{\mathop{\mathcal{G}_{[m]}}\nolimits}
\def\graphsN{\mathop{\mathcal{G}_{\mathbb{N}}}\nolimits}

\def\limitdensities{\mathop{\mathcal{D}^*}\nolimits}
\def\limitwiredensities{\mathop{\mathcal{V}^*}\nolimits}
\def\ind{\mathop{\text{ind}}\nolimits}

\newcommand{\xnorm}[1]{ \Vert #1 \Vert }

\def\part{\mathop{\mbox{Part}}\nolimits}

\def\ell{\mathop{[l]}\nolimits}
\def\equalinlaw{\mathop{=_{\mathcal{L}}}\nolimits}
\def\Nb{\mathop{\mathbb{N}_{}}\nolimits}

\newtheorem{thm}{Theorem}[section]

\newtheorem{lemma}[thm]{Lemma}
\newtheorem{prop}[thm]{Proposition}
\newtheorem{cor}[thm]{Corollary}
\newtheorem{defn}[thm]{Definition}
\newtheorem{example}[thm]{Example}

\newtheorem{rmk}[thm]{Remark}%\endlocaldefs

\title{Exchangeable Markov processes on graphs: Feller case}
\author{Harry Crane}\address{Rutgers University\\ Department
of Statistics \\ 
110 Frelinghuysen Road\\
Piscataway, NJ 08854\\
hcrane@stat.rutgers.edu.}
\date{\today}
\thanks{The author is partially supported by NSF grant DMS-1308899 and NSA grant H98230-13-1-0299.}
\begin{document}
\maketitle
\begin{abstract}
The transition law of every exchangeable Feller process on the space of countable graphs is determined by a $\sigma$-finite measure on the space of $\{0,1\}\times\{0,1\}$-valued arrays. 
In discrete-time, this characterization amounts to a construction from an independent, identically distributed sequence of exchangeable random functions.
In continuous-time, the behavior is enriched by a L\'evy--It\^o-type decomposition of the jump measure into mutually singular components that govern global, vertex-level, and edge-level dynamics.
Furthermore, every such process almost surely projects to a Feller process in the space of graph limits.  
\end{abstract}

\section{Introduction}\label{section:introduction}

A {\em graph}, or {\em network}, $G=(V,E_G)$ is a set of {\em vertices} $V$ and a binary relation of {\em edges} $E_G\subseteq V\times V$.
For $i,j\in V$, we write $G^{ij}$ to indicate the status of edge $(i,j)$, i.e.,
\[G^{ij}:=\mathbf{1}\{(i,j)\in E_G\}:=\left\{\begin{array}{cc}
1,& (i,j)\in E_G,\\
0,& \text{otherwise.}
\end{array}\right.\]
%Both the pair $G=(V,E_G)$ and the array $G=(G^{ij})_{i,j\in V}$ represent the same object.
We write $\mathcal{G}_V$ to denote the space of graphs with vertex set $V$.

Networks represent interactions among individuals, particles, and variables throughout science.
In this setting, consider a population of individuals labeled distinctly in $V$ and related to one another by the edges of a graph $G=(V,E_G)$.
The population size is typically large, but unknown, and so we assume a countably infinite population and take $V=\Nb:=\{1,2,\ldots\}$.
In practice, the labels $\Nb$ are arbitrarily assigned for the purpose of distinguishing individuals, and data can only be observed for a finite sample $S\subset\Nb$.
Thus, we often take $S=[n]:=\{1,\ldots,n\}$  whenever $S\subseteq\Nb$ is finite with cardinality $n\geq1$.

These practical matters relate to the operations of 
\begin{itemize}
	\item {relabeling}: the {\em relabeling} of $G=(G^{ij})_{i,j\geq1}$ by any permutation $\sigma:\Nb\rightarrow\Nb$ is
\begin{equation}\label{eq:relabeling}
G^{\sigma}:=(G^{\sigma(i)\sigma(j)})_{i,j\geq1},\quad\text{and}
\end{equation}
	\item {restriction}: the {\em restriction} of $G=(G^{ij})_{i,j\geq1}$ to a graph with vertices $S\subseteq\Nb$ is 
\begin{equation}\label{eq:restriction}
	G^S=G|_{S}:=(G^{ij})_{i,j\in S}.
\end{equation}
\end{itemize}
In many relevant applications, the network structure changes over time, resulting in a time-indexed collection $(G_t)_{t\in T}$ of graphs.
We consider exchangeable, consistent Markov processes as statistical models in this general setting.

\subsection{Graph-valued Markov processes}

A {\em $\graphsN$-valued Markov process} is a collection $\Gbf=\{\Gbf_G:\,G\in\graphsN\}$ for which each $\Gbf_G=(\Gamma_t)_{t\in T}$ is a family of random graphs satisfying $\Gamma_0=G$ and 
\begin{itemize}
	\item the {\em Markov property}, i.e., the past $(\Gamma_s)_{s<t}$ and future $(\Gamma_s)_{s>t}$ are conditionally independent given the present $\Gamma_t$, for all $t\in T$.
\end{itemize}
In addition, we assume $\Gbf$ is
\begin{itemize}
	\item {\em exchangeable}, i.e., all $\Gbf_G$ share a common {\em exchangeable} transition law such that
\begin{equation}\label{eq:exch tps}
	\mathbb{P}\{\Gamma_{t'}\in\cdot\mid\Gamma_t=F\}=\mathbb{P}\{\Gamma_{t'}^{\sigma}\in\cdot\mid\Gamma_t^{\sigma}=F\},\quad F\in\graphsN,\quad t'>t,\quad\text{ and }
\end{equation}
	\item {\em consistent}, i.e., the transition law of $\Gbf$ satisfies
\begin{equation}\label{eq:consistent}
	\mathbb{P}\{\Gamma_{t'}^{[n]}=F\mid\Gamma_t=F'\}=\mathbb{P}\{\Gamma_{t'}^{[n]}=F\mid\Gamma_t=F''\},\quad F\in\graphsn,\end{equation}
for all $F',F''\in\graphsN$ such that $F'|_{[n]}=F''|_{[n]}$, for every $n\in\Nb$.
\end{itemize}
Consistency implies that $\Gbf_G^{[n]}=(\Gamma_t|_{[n]})_{t\in T}$ satisfies the Markov property for every $n\in\Nb$, for all $G\in\graphsN$.
We call any $\Gbf$ satisfying these properties an {\em exchangeable, consistent Markov process}.
In Proposition \ref{prop:Feller equiv}, we observe that consistency and the Feller property are equivalent for exchangeable Markov processes on $\graphsN$, so we sometimes also call $\Gbf$ an {\em exchangeable Feller process}.

The process $\Gbf=\{\Gbf_G:\,G\in\graphsN\}$ is enough to determine the law of any collection $(\Gamma_t)_{t\in T}$ with a given transition law and initial distribution $\nu$ by first drawing $\Gamma_0\sim\nu$ and then putting $\Gbf_{\nu}=\Gbf_G$ on the event $\Gamma_0=G$.
With the underlying process $\Gbf=\{\Gbf_G:\,G\in\graphsN\}$ understood, we write $\Gbf_{\nu}=(\Gamma_t)_{t\in T}$ to denote a collection with this description.

Our main theorems characterize the behavior of exchangeable Feller processes in both discrete- and continuous-time.  
In discrete-time, we show that every exchangeable Feller process can be constructed by an iterated application of independent, identically distributed exchangeable random functions $\graphsN\rightarrow\graphsN$, called {\em rewiring maps}; see Section \ref{section:wiring maps}.
In continuous-time, $\Gbf$ admits a construction from a Poisson point process whose intensity measure has a L\'evy--It\^o-type decomposition.
The L\'evy--It\^o representation classifies every discontinuity of $\Gbf$ as one of three types.
In addition, when $\nu$ is an exchangeable initial distribution on $\graphsN$, both discrete- and continuous-time processes $\Gbf_{\nu}$ project to a Feller process in the space of graph limits.
These outcomes invoke connections to previous work on the theory combinatorial stochastic processes \cite{Bertoin2006,Pitman2005}, partially exchangeable arrays \cite{Aldous1981a,AldousExchangeability,Hoover1979}, and limits of dense graph sequences \cite{LovaszBook,LovaszSzegedy2006}.

%For every $n\in\Nb$ and $G\in\graphsN$, let $\Gbf_G^{[n]}=(\Gamma_t^{[n]})_{t\in T}$ be the restriction of $\Gbf_G$ to $\graphsn$, with $\Gamma_0^{[n]}=G|_{[n]}$.
%Consistency of $\Gbf$ implies that $\Gbf_G^{[n]}\equalinlaw\Gbf_{G'}^{[n]}$ for every $n\in\Nb$ and $G,G'\in\graphsN$ such that $G|_{[n]}=G'|_{[n]}$, where $\equalinlaw$ denotes %{\em equality in law}.
%Thus, $\Gbf=\{\Gbf_G:\,G\in\graphsN\}$ determines an exchangeable Markov process $\Gbf^{[n]}:=\{\Gbf_{G_n}:\,G_n\in\graphsn\}$ on $\graphsn$ for every $n\in\Nb$.

\subsection{Outline}

Before summarizing our conclusions (Section \ref{section:summary}), we first introduce key definitions and assumptions (Section \ref{section:preliminaries1}).
We then unveil essential concepts in more detail (Section \ref{section:preliminaries}) and prove our main theorems in discrete-time (Section \ref{section:discrete}) and in continuous-time (Section \ref{section:continuous}).  
We highlight some immediate extensions of our main theorems in our concluding remarks (Section \ref{section:concluding remarks}).

\section{Definitions and assumptions}\label{section:preliminaries1}

\subsection{Graphs}\label{section:graphs}
For any graph $G=(V,E_G)$, we impose the additional axioms of 
\begin{itemize}
	\item[(i)] {\em anti-reflexivity}, $(i,i)\notin E_G$ for all $i\in V$, and
	\item[(ii)] {\em symmetry}, $(i,j)\in E_G$ implies $(j,i)\in E_G$ for all $i,j\in V$.
\end{itemize}
Thus, we specialize $\mathcal{G}_V$ to denote the set of all graphs satisfying (i) and (ii).

By the symmetry axiom (ii), we write $ij=\{i,j\}\in E_G$ to indicate that there is an edge between $i$ and $j$ in $G$.
 By definition, $G$ has no multiple edges, condition (i) forbids edges from a vertex to itself, and condition (ii) makes $G$ {\em undirected}.   
In terms of the {\em adjacency array} $G=(G^{ij})_{i,j\in V}$, (i) and (ii) above correspond to
\begin{itemize}
	\item[(i')] $G^{ii}=0$ for all $i\in V$ and
	\item[(ii')] $G^{ij}=G^{ji}$ for all $i,j\in V$,
\end{itemize}
respectively.
As we discuss in Section \ref{section:concluding remarks}, our main theorems remain valid under some relaxation of each of these conditions.
The above two notions of a graph---as a pair $(V,E_G)$ and as an adjacency array---are equivalent; we use them interchangeably and with the same notation.   

With $\mathfrak{S}_V$ denoting the space of permutations of $V\subseteq\Nb$, i.e., bijective functions $\sigma:V\rightarrow V$, relabeling \eqref{eq:relabeling} associates every $\sigma\in\mathfrak{S}_V$ to a map $\mathcal{G}_V\rightarrow\mathcal{G}_V$.
For all $S\subseteq S'\subseteq V$, restriction \eqref{eq:restriction} determines a map $\mathcal{G}_{S'}\rightarrow\mathcal{G}_S$, $G\mapsto G^S=G|_S$.
Specifically, for $n\geq m\geq1$, $G|_{[m]}=(G^{ij})_{1\leq i,j\leq m}$ is the leading $m\times m$ submatrix of $G=(G^{ij})_{1\leq i,j\leq n}$.
By combining \eqref{eq:relabeling} and \eqref{eq:restriction}, every  injection $\phi:S\rightarrow S'$ determines a projection $\mathcal{G}_{S'}\rightarrow\mathcal{G}_{S}$ by
\begin{equation}\label{eq:phi-image}
G\mapsto G^{\phi}:=(G^{\phi(i)\phi(j)})_{i,j\in S}.
\end{equation}

We call a sequence of finite graphs $(G_n)_{n\in\mathbb{N}}$ {\em compatible} if $G_n\in\graphsn$ and  $G_n|_{[m]}=G_m$ for every $m\leq n$, for all $n\in\Nb$. 
Any compatible sequence of finite graphs determines a unique {\em countable graph} $G_{\infty}$, the projective limit of $(G_1,G_2,\ldots)$ under restriction.  
We endow $\graphsN$ with the product-discrete topology induced, for example, by the ultrametric
\begin{equation}\label{eq:metric}
d(G,G'):=1/\max\{n\in\mathbb{N}:G|_{[n]}=G'|_{[n]}\},\quad G,G'\in\graphsN.\end{equation}
From \eqref{eq:metric}, we naturally equip $\graphsN$ with the Borel $\sigma$-field $\sigma\langle\,\cdot|_{[n]}\rangle_{n\in\Nb}$ generated by the restriction maps.
Under \eqref{eq:metric}, $\graphsN$ is a compact and, therefore, complete and separable metric space; hence, $\graphsN$ is Polish and standard measure-theoretic outcomes apply in our analysis.

\subsection{Graph limits}\label{section:graph limits}

For $n\geq m\geq1$, $F\in\mathcal{G}_{[m]}$, and $G\in\graphsn$, we define the {\em density of $F$ in $G$} by
\begin{equation}\label{eq:induced density}
\delta(F,G):={\ind(F,G)}/{n^{\downarrow m}},\end{equation}
where $\ind(F,G)$ is the number of embeddings of $F$ into $G$ and $n^{\downarrow m}:=n(n-1)\cdots(n-m+1)$.  Specifically, 
\[\ind(F,G):=\sum_{\phi:[m]\rightarrow[n]}\mathbf{1}\{G^{\phi}=F\}\]
is the number of injections $\phi:[m]\rightarrow[n]$ for which $G^{\phi}=F$. 

Given a sequence $G=(G_n)_{n\in\Nb}$ in $\mathcal{G}^*:=\bigcup_{n\in\mathbb{N}}\graphsn$, we define the {\em limiting density of $F$ in $G$} by
\[
\delta(F,G):=\lim_{n\rightarrow\infty}\delta(F,G_n),\quad\text{if it exists}.\]  In particular, for $G\in\graphsN$, we define the {\em limiting density of $F$ in $G$} by
\begin{equation}\label{eq:limiting density F}
\delta(F,G):=\lim_{n\rightarrow\infty}\delta(F,G|_{[n]}),\quad\text{if it exists}.\end{equation}

\begin{defn}[Graph limit]\label{defn:graph limit}
The {\em graph limit} of $G\in\graphsN$ is the collection 
\[|G|:=(\delta(F,G))_{F\in\bigcup_{m\in\Nb}\graphsm},\]
 provided $\delta(F,G)$ exists for every $F\in\bigcup_{m\in\Nb}\graphsm$.
If $\delta(F,G)$ does not exist for some $F$, then we put $|G|:=\partial$.
We write $\limitdensities$ to denote the closure of $\{|G|:\,G\in\graphsN\}\setminus\{\partial\}$ in $[0,1]^{\bigcup_{m\in\Nb}\graphsm}$.
\end{defn}

If the graph limit of $G$ exists, then $G$ determines a family of exchangeable probability distributions on $(\graphsn)_{n\in\Nb}$ by
\begin{equation}\label{eq:graph limit dist}\mathbb{P}\{\Gamma_n=F\mid G\}=\delta(F,G),\quad F\in\graphsn,\quad n\in\Nb.\end{equation}
Moreover, the distributions determined by $(\delta(F,G))_{F\in\mathcal{G}_{[m]}}$ and $(\delta(F,G))_{F\in\graphsn}$, $m\leq n$, are mutually {\em consistent}, i.e., the distribution in \eqref{eq:graph limit dist} satisfies
\[\mathbb{P}\{\Gamma_n|_{[m]}=F'\mid G\}=\sum_{F\in\graphsn:\,F|_{[m]}=F'}\delta(F,G)=\delta(F',G),\quad\text{for every } F'\in\mathcal{G}_{[m]}.\]
Thus, every $D\in\limitdensities$ determines a unique probability measure on $\graphsN$, which we denote by $\gamma_D$.
In addition, $\gamma_D$ is {\em exchangeable} in the sense that $\Gamma\sim\gamma_D$ satisfies $\Gamma^{\sigma}\equalinlaw\Gamma$ for all $\sigma\in\symmetricN$, where $\equalinlaw$ denotes {\em equality in law}.

\begin{rmk}
In our notation, $D\in\limitdensities$ and $\gamma_D$ correspond to the same object, but we use the former to emphasize that $D$ is the graph limit of some $G\in\graphsN$ and the latter to specifically refer to the probability distribution that $D$ determines on $\graphsN$.
As an element of $\limitdensities$, $D=(D_F)_{F\in\mathcal{G}^{*}}$ is an element of $[0,1]^{\bigcup_{m\in\Nb}\graphsm}$.
The connection between $D$ and $\gamma_D$ is made explicit by
\[D_F=D(F)=\gamma_D(\{G\in\graphsN:\,G|_{[n]}=F\}),\quad F\in\graphsn,\quad n\in\Nb.\]
\end{rmk}

\begin{defn}[Dissociated random graphs]\label{defn:dissociated graph}
A random graph $\Gamma$ is {\em dissociated} if
\begin{equation}\label{eq:dissociated graph}
\Gamma|_{S}\quad\text{and}\quad\Gamma|_{S'}\quad\text{ are independent for all disjoint subsets }S,S'\subseteq\Nb.
\end{equation}
We call a probability measure $\nu$ on $\graphsN$ {\em dissociated} if $\Gamma\sim\nu$ is a dissociated random graph.
\end{defn}
Dissociated measures play a central role in the theory of exchangeable random graphs:
the measure $\gamma_D$ determined by any $D\in\limitdensities$ is dissociated, and conversely the Aldous--Hoover theorem (Theorem \ref{thm:A-H}) states that every exchangeable probability measure on $\graphsN$ is a mixture of exchangeable, dissociated probability measures.
In particular, to every exchangeable probability measure $\nu$ on $\graphsN$ there exists a unique probability measure $\Delta$ on $\limitdensities$ such that $\nu=\gamma_{\Delta}$, where
\begin{equation}\label{eq:Delta-mixture}
\gamma_{\Delta}(\cdot):=\int_{\limitdensities}\gamma_D(\cdot)\Delta(dD)\end{equation}
is the mixture of $\gamma_D$ measures with respect to $\Delta$.

\subsection{Notation}

We  use the capital Roman letter $G$ to denote a generic graph, capital Greek letter $\Gamma$ to denote a random graph, bold Greek letter with subscript $\mathbf{\Gamma}_{\bullet}$ to denote a collection of random graphs with initial condition $\bullet$, and bold Greek letter $\Gbf=\{\Gbf_{\bullet}\}$ to denote a graph-valued process indexed by the initial conditions $\bullet$.

Superscripts index edges and subscripts index time; therefore, for $(\Gamma_t)_{t\in T}$, we write $\Gamma_t^{ij}$ to indicate the status of edge $ij$ at time $t\in T$ and $\Gamma^{S}_{T'}=(\Gamma_t^{S})_{t\in T'}$ to denote the trajectory of the edges $ij\subseteq S\subseteq\Nb$ over the set of times in $T'\subseteq T$.
We adopt the same notation for processes $\mathbf{\Gamma}=\{\Gbf_G:\,G\in\graphsN\}$, with $\mathbf{\Gamma}_{T'}^{S}:=\{\mathbf{\Gamma}_{G,T'}^S:\,G\in\graphsN\}$ denoting the restriction of $\Gbf$ to edges $ij\in S$ and times $T'\subseteq T$.

We distinguish between discrete-time ($T=\mathbb{Z}_+:=\{0,1,2,\ldots\}$) and continuous-time ($T=\mathbb{R}_+:=[0,\infty)$) processes by indexing time by $m$ and $t$, respectively; thus, $(\Gamma_m)_{m\geq0}$ denotes a discrete-time process and $(\Gamma_t)_{t\geq0}$ denotes a continuous-time process.
To further distinguish these cases, we call $\Gbf$ a {\em Markov chain} if its constituents are indexed by discrete-time and a {\em Markov process} if its constituents are indexed by continuous-time.
Whenever a discussion encompasses both discrete- and continuous-time, we employ terminology and notation for the continuous-time case.

\section{Summary of main theorems}\label{section:summary}

We now state our main theorems, saving many technical details for later.  
Of primary importance are the notions of {rewiring maps} and {rewiring limits}, which we introduce briefly in Section \ref{section:discrete-intro} and develop formally in Section \ref{section:wiring maps}.

\subsection{Discrete-time Feller chains}\label{section:discrete-intro}
Let $W=(W^{ij})_{i,j\in V}$ be a symmetric $\{0,1\}\times\{0,1\}$-valued array with $W^{ii}=(0,0)$ for all $i\in V$.
For $i,j\in V$, we express the $(i,j)$ component of $W$ as a pair $W^{ij}=(W^{ij}_0,W^{ij}_1)$.  
For any such $W$ and any $G\in\mathcal{G}_V$, we define the {\em rewiring of $G$ by $W$} by $G'=W(G)$, where
\begin{equation}\label{eq:wire G}
G'^{ij}:=\left\{\begin{array}{cc}
W_0^{ij},& G^{ij}=0\\
W_1^{ij},& G^{ij}=1\end{array}\right.,\quad i,j\in V.\end{equation}
Note that $W$ acts on $G$ by reconfiguring its adjacency array at each entry: if $G^{ij}=0$, then $G'^{ij}$ is taken from $W_0^{ij}$; otherwise, $G'^{ij}$ is taken from $W_1^{ij}$.  
Thus, we can regard $W$ as a function $\mathcal{G}_V\rightarrow\mathcal{G}_V$, called a {\em rewiring map}.
Lemma \ref{lemma:existence limit density} records some basic facts about rewiring maps.

Our definition of relabeling and restriction for rewiring maps is identical to definitions \eqref{eq:relabeling} and \eqref{eq:restriction} for graphs, i.e., for every $\sigma\in\mathfrak{S}_V$ and $S\subseteq V$, we define
\begin{align*}
	W^{\sigma}&:=(W^{\sigma(i)\sigma(j)})_{i,j\in V}\quad\text{and}\\
	W^S=W|_{S}&:=(W^{ij})_{i,j\in S}.
\end{align*}
We write $\mathcal{W}_V$ to denote the collection of rewiring maps $\mathcal{G}_V\rightarrow\mathcal{G}_V$.  

Given a probability measure $\omega$ on $\wireN$, we construct a discrete-time Markov chain $\mathbf{\Gamma}^*_{\omega}=\{\Gbf^*_{G,\omega}:\,G\in\graphsN\}$ as follows. 
We first generate $W_1,W_2,\ldots$ independent and identically distributed (i.i.d.)~according to $\omega$.
For every $G\in\graphsN$, we define $\Gbf^*_{G,\omega}=(\Gamma^*_m)_{m\geq0}$ by putting $\Gamma^*_0=G$ and
\begin{equation}\label{eq:iterated construction}
\Gamma^*_m:=W_m(\Gamma^*_{m-1})=(W_m\circ\cdots\circ W_1)(G),\quad m\geq1,
\end{equation}
where $W\circ W'$ indicates the composition of $W$ and $W'$ as maps $\graphsN\rightarrow\graphsN$.

From the definition in \eqref{eq:wire G}, $\mathbf{\Gamma}_{\omega}^*$ in \eqref{eq:iterated construction} is a consistent Markov chain.
If, in addition, $W\sim\omega$ satisfies $(W^{ij})_{i,j\geq1}\equalinlaw(W^{\sigma(i)\sigma(j)})_{i,j\geq1}$ for all permutations $\sigma:\Nb\rightarrow\Nb$, then $\Gbf_{\omega}^*$ is also exchangeable.
Theorem \ref{thm:discrete char}  establishes the converse: to every discrete-time exchangeable, consistent Markov chain $\Gbf=\{\Gbf_G:\,G\in\graphsN\}$ there corresponds a probability measure $\omega$ on $\wireN$ such that $\Gbf\equalinlaw\Gbf^*_{\omega}$, that is, $\Gbf_G\equalinlaw\Gbf^*_{G,\omega}$ for all $G\in\graphsN$.

By appealing to the theory of partially exchangeable arrays, we further characterize $\omega$ uniquely in terms of a mixing measure $\Upsilon$ on the space of {\em rewiring limits}, which we express as $\omega=\Omega_{\Upsilon}$.
We present these outcomes more explicitly in Section \ref{section:wiring maps}.

\begin{thm}[Discrete-time characterization]\label{thm:discrete char}
Let $\Gbf=\{\Gbf_G:\,G\in\graphsN\}$ be a discrete-time, exchangeable, consistent Markov chain on $\graphsN$.
Then there exists a probability measure $\Upsilon$ such that $\mathbf{\Gamma}\equalinlaw\mathbf{\Gamma}^*_{\Upsilon}$, where $\mathbf{\Gamma}^*_{\Upsilon}$ is constructed as in \eqref{eq:iterated construction} from $W_1,W_2,\ldots$ independent and identically distributed according to $\omega=\Omega_{\Upsilon}$. 
\end{thm}

Our proof of Theorem \ref{thm:discrete char} relies on an extension (Theorem \ref{thm:discrete-rep}) of the Aldous--Hoover theorem (Theorem \ref{thm:A-H}) to the case of exchangeable and consistent Markov chains on $\{0,1\}$-valued arrays.
Heuristically, $W_1,W_2,\ldots$  must be i.i.d.\ because of the Markov and consistency assumptions:
by consistency, every finite restriction $\mathbf{\Gamma}^{[n]}_G$ has the Markov property and, therefore, for every $m\geq1$ the conditional law of $W_m$, given $\Gamma_{m-1}$, must satisfy 
\[W_m^{[n]}\text{ and }\Gamma_{m-1}^{\Nb\setminus[n]}\text{ are independent for every }n\in\Nb;\]
exchangeability further requires that $W_m^{[n]}$ is independent of $\Gamma_{m-1}^{[n]}$, and so we expect  $W_m$ to be independent of $\Gamma_{m-1}$ for every $m\geq1$;
by time-homogeneity, $W_1,W_2,\ldots$ should also be identically distributed;
independence of $W_1,W_2,\ldots$ then follows from the Markov property.  
Each of these above points on its own requires a nontrivial argument.
We make this heuristic rigorous in Section \ref{section:discrete}.

\subsubsection{Projection into the space of graph limits}\label{section:discrete graph limit}
%Another byproduct of de Finetti's theorem is the almost sure existence of the random variable
%\[|Z|:=\lim_{n\rightarrow\infty}n^{-1}\sum_{j=1}^{n}\mathbf{1}\{Z^j=1\},\]
%the limiting proportion of ones in an exchangeable $\{0,1\}$-valued sequence $Z=(Z^1,Z^2,\ldots)$.
%Existence of $|Z|$ is guaranteed by combining the law of large numbers with de Finetti's representation of $Z$ as a mixture of i.i.d.\ sequences.
%Since the unit interval is in correspondence with the collection of i.i.d.\ probability measures on $\{0,1\}$-valued sequences, $|Z|$ determines a random probability measure on i.i.d.\ %sequences.
From our discussion of graph limits in Section \ref{section:graph limits}, any exchangeable random graph $\Gamma$ projects almost surely to a {\em graph limit} $|\Gamma|=D$, which corresponds to a random element $\gamma_{D}$ in the space of exchangeable, dissociated probability measures on $\graphsN$.
Likewise, any exchangeable rewiring map $W\in\wireN$  projects to a {\em rewiring limit} $|W|=\upsilon$, which corresponds to an exchangeable, dissociated  probability measure $\Omega_{\upsilon}$ on $\wireN$.
We write $\limitdensities$ and $\limitwiredensities$ to denote the spaces of graph limits and rewiring limits, respectively.

We delay more formal details about  rewiring limits until Section \ref{section:wiring maps}.  
For now, we settle for an intuitive understanding by analog to graph limits.
Recall the definition of $\gamma_{\Delta}$ in \eqref{eq:Delta-mixture} for a probability measure $\Delta$ on $\limitdensities$, i.e., $\Gamma\sim\gamma_{\Delta}$ is a random graph obtained by first sampling $D\sim\Delta$ and, given $D$, drawing $\Gamma$ from $\gamma_D$.
Similarly, for a probability measure $\Upsilon$ on $\limitwiredensities$, $W\sim\Omega_{\Upsilon}$ is a random rewiring map obtained by first sampling $\upsilon\sim\Upsilon$ and, given $\upsilon$, drawing $W$ from $\Omega_\upsilon$.
In particular, $W\sim\Omega_{\Upsilon}$ implies $|W|\sim\Upsilon$.
Just as for graph limits, we regard $\upsilon\in\limitwiredensities$ and $\Omega_{\upsilon}$ as the same object, but use the former to refer to a rewiring limit of some $W\in\wireN$ and the latter to refer to the corresponding exchangeable, dissociated probability distribution on $\wireN$.

From the construction of $\Gbf^*_{\omega}$ in \eqref{eq:PPP construction}, every probability measure $\omega$ on $\wireN$ determines a transition probability $P_{\omega}(\cdot,\cdot)$ on $\graphsN$ by
\begin{equation}\label{eq:induced tps}
P_{\omega}(G,\cdot):=\omega(\{W\in\wireN:W(G)\in\cdot\}),\quad G\in\graphsN,\end{equation}
and thus every $\upsilon\in\limitwiredensities$ determines a transition probability $P_{\upsilon}(\cdot,\cdot)$ by taking $\omega=\Omega_{\upsilon}$ in \eqref{eq:induced tps}.
Consequently, every $\upsilon\in\limitwiredensities$ acts on $\limitdensities$ by composition of probability measures, 
\begin{equation}\label{eq:right action}
D\mapsto D\upsilon(\cdot):=\int_{\graphsN}P_\upsilon(G,\cdot)\gamma_D(dG).\end{equation}
In other words,  $D\upsilon$ is the probability measure of $\Gamma'$ obtained by first generating $\Gamma\sim\gamma_D$ and, given $\Gamma=G$, taking $\Gamma'=W(G)$, where $W\sim\Omega_{\upsilon}$ is a random rewiring map.

Alternatively, we can express $D$ as a countable vector $(D(F))_{F\in\mathcal{G}^*}$ with
\[D(F):=\gamma_D(\{G\in\graphsN:\,G|_{[n]}=F\}),\quad F\in\graphsn,\quad n\in\Nb.\]
Moreover, $\upsilon$ gives rise to an infinite by infinite array $(\upsilon(F,F'))_{F,F'\in\mathcal{G}^*}$ with
\[\upsilon(F,F'):=\left\{\begin{array}{cc} \Omega_{\upsilon}(\{W\in\wireN:\,W|_{[n]}(F)=F'\}),& F,F'\in\graphsn\text{ for some }n\in\Nb,\\
0,& \text{otherwise.}
\end{array}\right.\]
In this way, $(\upsilon(F,F'))_{F,F'\in\mathcal{G}^*}$ has finitely many non-zero entries in each row and the usual definition of $D\upsilon=(D'(F))_{F\in\mathcal{G}^*}$ by
\begin{equation}\label{eq:right action-2}
D'(F):=\sum_{F'\in\graphsn}D(F')\upsilon(F',F),\quad F\in\graphsn,\quad n\in\Nb,\end{equation}
coincides with \eqref{eq:right action}.
(Note that the array $(\upsilon(F,F'))_{F,F'\in\mathcal{G}^*}$ does not determine $\upsilon$, but its action on $\limitdensities$ through \eqref{eq:right action-2} agrees with the action of $\upsilon$ in \eqref{eq:right action}.)

For a graph-valued process $\Gbf=\{\Gbf_G:\,G\in\graphsN\}$, we write $\Gbf_{\limitdensities}=\{\Gbf_{D}:\,D\in\limitdensities\}$ to denote the process derived from $\Gbf$ by defining the law of $\Gbf_D=(\Gamma_t)_{t\in T}$ as that of the collection of random graphs obtained by taking $\Gamma_0\sim\gamma_D$ and, given $\Gamma_0=G$, putting $\Gbf_D=\Gbf_G$.
Provided $|\Gamma_t|$ exists for all $t\in T$, we define $|\Gbf_D|:=(|\Gamma_t|)_{t\in T}$ as the collection of graph limits at all times of $\Gbf_D$.
If $|\Gbf_D|$ exists for all $D\in\limitdensities$, we write $|\Gbf_{\limitdensities}|:=\{|\Gbf_D|:\,D\in\limitdensities\}$ to denote the associated $\limitdensities$-valued process.

\begin{thm}[Graph limit chain]\label{thm:induced density chain}
Let $\mathbf{\Gamma}=\{\Gbf_G:\,G\in\graphsN\}$ be a discrete-time, exchangeable Feller chain on $\graphsN$.
Then $|\Gbf_{\limitdensities}|$  exists almost surely and is a Feller chain on $\limitdensities$.  
Moreover, $|\Gbf_{D}|\equalinlaw\mathbf{D}^*_{D,\Upsilon}$ for every $D\in\limitdensities$, where $\mathbf{D}^*_{D,\Upsilon}:=(D^*_m)_{m\geq0}$ satisfies $D^*_0=D$ and
\begin{equation}\label{eq:iterated product}
D^*_m:=D^*_{m-1}{Y_m}=D{Y_1}\cdots {Y_m},\quad m\geq1,\end{equation}
with $Y_1,Y_2,\ldots$ i.i.d.\ from $\Upsilon$, the probability measure associated to $\Gbf$ through Theorem \ref{thm:discrete char}.
\end{thm}

\subsection{Continuous-time Feller processes}\label{section:continuous-intro}
In discrete time, we use an i.i.d.\ sequence of rewiring maps to construct the representative Markov chain $\mathbf{\Gamma}^*_{\Upsilon}$, cf.\ Equation \eqref{eq:iterated construction} and Theorem \ref{thm:discrete char}.
In continuous time, we construct a representative Markov process $\mathbf{\Gamma}^*_{\omega}=\{\Gbf^*_{G,\omega}:\,G\in\graphsN\}$ from a Poisson point process $\mathbf{W}=\{(t,W_t)\}\subset[0,\infty)\times\wireN$ with intensity $dt\otimes\omega$.
Here, $dt$ denotes Lebesgue measure on $[0,\infty)$ and $\omega$ satisfies
\begin{equation}\label{eq:regularity omega}
\omega(\{\id\})=0\quad\text{and}\quad\omega(\{W\in\wireN:W|_{[n]}\neq\idn\})<\infty\text{ for every }n\in\mathbb{N},\end{equation}
with $\mathbf{Id}_V$ denoting the identity $\mathcal{G}_V\rightarrow\mathcal{G}_V$, $V\subseteq\Nb$.

We construct each $\Gbf_{G,\omega}^*$ through its finite restrictions $\mathbf{\Gamma}^{*[n]}_{G,\omega}:=(\Gamma^{*[n]}_t)_{t\geq0}$ on $\graphsn$ by first putting $\Gamma^{*[n]}_0:=G|_{[n]}$ and then defining
\begin{equation}\label{eq:PPP construction}
\begin{array}{l}
\bullet\quad \Gamma^{*[n]}_t:=W_t|_{[n]}(\Gamma^{*[n]}_{t-}),\text{ if }t>0\text{ is an atom time of }\mathbf{W}\text{ for which }W_t|_{[n]}\neq\idn, \text{ or}\\
\bullet\quad \Gamma^{*[n]}_{t}:=\Gamma^{*[n]}_{t-},\text{ otherwise}.
\end{array}
\end{equation}
Above, we have written $\Gamma_{t-}^{*[n]}:=\lim_{s\uparrow t}\Gamma_{s}^{*[n]}$ to denote the state of $\mathbf{\Gamma}^{*[n]}_{G,\omega}$ in the instant before time $t>0$.

By \eqref{eq:regularity omega}, every $\mathbf{\Gamma}^{*[n]}_{G,\omega}$ has c\`adl\`ag paths and is a Markov chain on $\graphsn$.  
Moreover, if $\omega$ is exchangeable, then so is $\mathbf{\Gamma}^{*[n]}_{\omega}$. 
By construction, $\mathbf{\Gamma}^{*[m]}_{G,\omega}$ is the restriction to $\mathcal{G}_{[m]}$ of $\mathbf{\Gamma}^{*[n]}_{G,\omega}$, for every $m\leq n$ and every $G\in\graphsN$.
In particular, $\Gbf^{*[n]}_{G,\omega}=\Gbf^{*[n]}_{G',\omega}$ for all $G,G'\in\graphsN$ such that $G|_{[n]}=G'|_{[n]}$.
Thus, $(\mathbf{\Gamma}^{*[n]}_{G,\omega})_{n\in\Nb}$ determines a unique collection $\mathbf{\Gamma}^*_{G,\omega}=(\Gamma^*_t)_{t\geq0}$ in $\graphsN$, for every $G\in\graphsN$.
We define the process $\Gbf^*_{\omega}=\{\Gbf_{G,\omega}^{*}:\,G\in\graphsN\}$ as the family of all such collections generated from the same Poisson point process.

\begin{thm}[Poissonian construction]\label{thm:Poisson}
Let $\mathbf{\Gamma}=\{\Gbf_G:\,G\in\graphsN\}$ be an exchangeable Feller process on $\graphsN$. 
Then there exists an exchangeable measure $\omega$ satisfying \eqref{eq:regularity omega} such that $\mathbf{\Gamma}\equalinlaw\Gbf_{\omega}^*$, as constructed in \eqref{eq:PPP construction}.
\end{thm}

\subsubsection{L\'evy--It\^o representation}

Our next theorem builds on Theorem \ref{thm:Poisson} by classifying the discontinuities of $\Gbf$ into three types.
If $s>0$ is a discontinuity time for $\Gbf_G=(\Gamma_t)_{t\geq0}$, then either
\begin{itemize}
	\item[(I)] there is a unique edge $ij$ for which $\Gamma_{s-}^{ij}\neq \Gamma^{ij}_s$,
	\item[(II)] there is a unique vertex $i\geq1$ for which the edge statuses $(\Gamma_{s-}^{ij})_{j\neq i}$ are updated according to an exchangeable transition probability on $\{0,1\}^{\mathbb{N}}$ and all edges not incident to $i$ remain fixed, or
	\item[(III)] the edges of $\Gamma_{s-}$ are updated according to a random rewiring map as in discrete-time.
\end{itemize}
Qualitatively, the above description separates the jumps of $\mathbf{\Gamma}_{G}$ according to their local and global characteristics.  
Type (I) discontinuities are local---they involve a change in status to just a single edge---while Type (III) discontinuities are global---they involve a change to a positive proportion of all edges.  
Type (II) discontinuities lie in between---they involve a change to infinitely many but a zero limiting proportion of edges.
% $\lim_{n\rightarrow\infty}2n/(n(n-1))=0$.  So, although an infinite number of edges change status in (II), the overall proportion of edges affected is zero.  
According to this characterization, there can be no discontinuities affecting the status of more than one but a zero limiting fraction of vertices or more than one but a zero limiting fraction of edges.  

The decomposition of the discontinuities into Types (I)-(III) prompts the {\em L\'evy--It\^o decomposition} of the intensity measure $\omega$ from Theorem \ref{thm:Poisson}.
More specifically, Theorem \ref{thm:Levy-Ito} below decomposes the jump measure of $\omega$ into a unique quadruple $(\mathbf{e},\mathbf{v},\Sigma,\Upsilon)$, where $\mathbf{e}=(\mathbf{e}_0,\mathbf{e}_1)$ is a unique pair of non-negative constants, $\mathbf{v}$ is a unique non-negative constant, $\Sigma$ is a unique probability measure on the space of $2\times 2$ stochastic matrices, and $\Upsilon$ is a unique measure on the space of rewiring limits.
These components contribute to the behavior of $\Gbf_G$, $G\in\graphsN$, as follows.
\begin{itemize}
	\item[(I)] For unique constants $\mathbf{e}_0,\mathbf{e}_1\geq0$, each edge $ij$, $i\neq j$, changes its status independently of all others: each edge in $\Gbf_G$ is removed independently at Exponential rate $\mathbf{e}_0\geq0$ and each absent edge in $\Gbf_G$ is added independently at Exponential rate $\mathbf{e}_1\geq0$.  A jump of this kind results in a  Type (I) discontinuity.
	\item[(II)] Each vertex jumps independently at Exponential rate $\mathbf{v}\geq0$.
When vertex $i\in\mathbb{N}$ experiences a jump, the statuses of edges $ij$, $j\neq i$, are updated conditionally independently according to a random transition probability matrix $S$ generated from a unique probability measure $\Sigma$ on the space of $2\times 2$ stochastic matrices:
\[\mathbb{P}\{\Gamma^{ij}_s=k'\mid\Gamma^{ij}_{s-}=k,\,S\}=S_{kk'},\quad k,k'=0,1,\]  
where $S=(S_{kk'})\sim \Sigma$.
Edges that do not involve the specified vertex $i$ stay fixed.
A jump of this kind results in a Type (II) discontinuity.
	\item[(III)] A unique measure $\Upsilon$ on $\limitwiredensities$ determines a transition rate measure $\Omega_{\Upsilon}$, akin to the induced transition probability measure \eqref{eq:induced tps} discussed in Section \ref{section:discrete graph limit}.
%This transition rate measure is akin to the induced transition probability measures in \eqref{eq:induced tps}.
%These jumps are akin to the jumps that occur in discrete-time, except $\Upsilon$ need not be finite.
A jump of this kind results in a Type (III) discontinuity.
\end{itemize}

For $i\in\mathbb{N}$ and any probability measure $\Sigma$ on $2\times 2$ stochastic matrices, we write $\Omega_{\Sigma}^{(i)}$ to denote the probability measure induced on $\wireN$ by the procedure applied to vertex $i$ in (II) above.
That is, we generate $W\sim\Omega_{\Sigma}^{(i)}$ by first taking $S\sim\Sigma$ and, given $S$, generating $W^{ij}=(W_0^{ij},W_1^{ij})$, $j\neq i$, as independent and identically distributed from 
\begin{align*}
\mathbb{P}\{W_0^{ij}=k\mid S\}&=S_{0k},\quad k=0,1\text{ and }j\neq i,\quad\text{and}\\
\mathbb{P}\{W_1^{ij}=k\mid S\}&=S_{1k},\quad k=0,1\text{ and }j\neq i.\end{align*}
For all other edges $i'j'$, $i'\neq j'$, such that $i\notin\{i',j'\}$, we put $W^{i'j'}=(0,1)$.
Thus, the action of $W\sim\Omega_{\Sigma}^{(i)}$ on $G\in\graphsN$ results in a discontinuity of Type (II).
By slight abuse of notation, we write 
\begin{equation}\label{eq:Omega_Sigma}
\Omega_{\Sigma}:=\sum_{i=1}^{\infty}\Omega_{\Sigma}^{(i)}.\end{equation}
For $k=0,1$, we define $\epsilon_{k}$ as the measure that assigns unit mass to each rewiring map with a single off-diagonal entry equal to $(k,k)$ and all other off-diagonal entries $(0,1)$.

In the next theorem, $\mathbf{I}:=|\id|$ denotes the rewiring limit of the identity $\id:\wireN\rightarrow\wireN$ and $\upsilon_*^{(2)}$ is the mass assigned to $\idtwo$ by $\Omega_{\upsilon}$, for any $\upsilon\in\limitwiredensities$.

\begin{thm}[L\'evy--It\^o representation]\label{thm:Levy-Ito}
Let $\Gbf_{\omega}^*=\{\Gbf_{G,\omega}^*:\,G\in\graphsN\}$ be an exchangeable Feller process constructed in \eqref{eq:PPP construction} based on intensity $\omega$ satisfying \eqref{eq:regularity omega}.
Then there exist unique constants $\mathbf{e}_0,\mathbf{e}_1,\mathbf{v}\geq0$, a unique probability measure $\Sigma$ on $2\times2$ stochastic matrices, and a unique measure $\Upsilon$ on $\limitwiredensities$ satisfying
\begin{equation}\label{eq:regularity upsilon00}
\Upsilon(\{\mathbf{I}\})=0\quad\text{and}\quad\int_{\limitwiredensities}(1-\upsilon_*^{(2)})\Upsilon(d\upsilon)<\infty
\end{equation}
such that
\begin{equation}\omega=\Omega_{\Upsilon}+\mathbf{v}\Omega_{\Sigma}+\mathbf{e}_0\epsilon_0+\mathbf{e}_1\epsilon_1.\label{eq:Levy-Ito measure00}\end{equation}
\end{thm}

\begin{rmk}
By Theorem \ref{thm:Poisson}, every exchangeable Feller process $\Gbf$ admits a construction by $\Gbf_{\omega}^*$ in \eqref{eq:PPP construction} for some $\omega$ satisfying \eqref{eq:regularity omega}.
Therefore, Theorem \ref{thm:Levy-Ito} provides a L\'evy--It\^o-type representation for all exchangeable $\graphsN$-valued Feller processes $\Gbf$.
\end{rmk}

\begin{rmk}
The  regularity condition \eqref{eq:regularity upsilon00} and characterization \eqref{eq:Levy-Ito measure00} recall a similar description for $\mathbb{R}^d$-valued L\'evy processes, which decompose as the superposition of independent Brownian motion with drift, a compound Poisson process, and a pure-jump martingale.  
In $\mathbb{R}^d$, the Poissonian component is characterized by a {\em L\'evy measure} $\Pi$ satisfying
\begin{equation}\label{eq:Levy regularity}
\Pi(\{(0,\ldots,0)\})=0\quad\text{and}\quad\int(1\wedge |x|^2)\Pi(dx)<\infty;\end{equation}
see Theorem 1 of \cite{BertoinLevy}.  
Only a little imagination is needed to appreciate the resemblance between \eqref{eq:regularity upsilon00} and \eqref{eq:Levy regularity}.
\end{rmk}

 Intuitively, \eqref{eq:regularity upsilon00} is the naive condition needed to ensure that $\Omega_{\Upsilon}$ satisfies \eqref{eq:regularity omega}.  By the Aldous--Hoover theorem (Theorem \ref{thm:A-H}), $\upsilon_*^{(2)}$ corresponds to the probability that a random rewiring map from $\Omega_{\upsilon}$ restricts to $\idtwo\in\mathcal{W}_2$; thus,
\[\int_{\limitwiredensities}(1-\upsilon_*^{(2)})\Upsilon(d\upsilon)<\infty\]
guarantees that the restriction of $\mathbf{\Gamma}$ to $\mathcal{G}_{[2]}$ has c\`adl\`ag sample paths.  
Under exchangeability, this is enough to guarantee that $\mathbf{\Gamma}^{[n]}$ has c\`adl\`ag paths for all $n\in\mathbb{N}$.

\subsubsection{Projection into the space of graph limits}

As in discrete-time, we write $\Gbf_{\limitdensities}$ for the process derived from $\Gbf$ by mixing with respect to each initial distribution $\gamma_D$, $D\in\limitdensities$.
Analogous to Theorem \ref{thm:induced density chain}, every exchangeable Feller process almost surely projects to a Feller process in the space of graph limits.

\begin{thm}[Graph limit process]\label{thm:induced density process}
Let $\mathbf{\Gamma}=\{\Gbf_G:\,G\in\graphsN\}$ be an exchangeable Feller process on $\graphsN$.
Then $|\mathbf{\Gamma}_{\limitdensities}|$ exists almost surely and is a Feller process on $\limitdensities$.
Moreover, each $|\mathbf{\Gamma}_D|$ is continuous at all $t>0$ except possibly at the times of Type (III) discontinuities.
\end{thm}

\section{Preliminaries: Graph-valued processes}\label{section:preliminaries}

Proof of our main theorems is spread over Sections \ref{section:discrete} and \ref{section:continuous}.
In preparation, we first develop the machinery of exchangeable random graphs, graph-valued processes, rewiring maps, graph limits, and rewiring limits.

\subsection{Exchangeable random graphs}\label{section:exchangeable random graphs}

\begin{defn}[Exchangeable random graphs]
A random graph $\Gamma\in\mathcal{G}_V$ is {\em exchangeable} if $\Gamma\equalinlaw\Gamma^{\sigma}$ for all $\sigma\in\mathfrak{S}_V$.
 A measure $\gamma$ on $\graphsN$ is {\em exchangeable} if $\gamma(S)=\gamma(S^{\sigma})$ for all $\sigma\in\mathfrak{S}_V$ and all measurable subsets $S\subseteq\graphsN$, where $S^{\sigma}:=\{G^{\sigma}:\,G\in S\}$.
\end{defn}

%\begin{example}[Erd\H{o}s--R\'enyi random graph \cite{ErdosRenyi1959}]\label{ex:Erdos-Renyi}

%\end{example}

In the following theorem, $\mathcal{X}$ is any Polish space.

\begin{thm}[Aldous--Hoover \cite{AldousExchangeability}, Theorem 14.21]\label{thm:A-H}
Let $X:=(X^{ij})_{i,j\geq1}$ be a symmetric $\mathcal{X}$-valued array for which $X\equalinlaw(X^{\sigma(i)\sigma(j)})_{i,j\geq1}$ for all $\sigma\in\symmetricN$.  Then there exists a measurable function $f:[0,1]^4\rightarrow\mathcal{X}$ such that $f(\cdot, b,c,\cdot)=f(\cdot,c,b,\cdot)$ and $X\equalinlaw X^*$, where $X^*:=(X^{*ij})_{i,j\geq1}$ is defined by
\begin{equation}\label{eq:AH-rep}
X^{*ij}:=f(\zeta_{\emptyset},\zeta_{\{i\}},\zeta_{\{j\}},\zeta_{\{i,j\}}),\quad i,j\geq 1,\end{equation}
for $\{\zeta_{\emptyset}; (\zeta_{\{i\}})_{i\geq1}; (\zeta_{\{i,j\}})_{j>i\geq1}\}$ i.i.d.\ Uniform random variables on $[0,1]$.
\end{thm}

By taking $\mathcal{X}=\{0,1\}$, the Aldous--Hoover theorem provides a general construction for all exchangeable random graphs with countably many vertices.
We make repeated use of the representation in \eqref{eq:AH-rep} throughout our discussion, with particular emphasis on the especially nice Aldous--Hoover representation and other special properties of  the Erd\H{o}s--R\'enyi process.

\subsection{Erd\H{o}s--R\'enyi random graphs}\label{section:ER}
For fixed $p\in[0,1]$, the {\em Erd\H{o}s--R\'enyi measure} $\varepsilon_p$ on $\graphsN$ is defined as the distribution of $\Gamma\in\graphsN$ obtained by including each edge independently with probability $p$.
The Erd\H{o}s--R\'enyi measure is exchangeable for every $p\in[0,1]$.
We call $\Gamma\sim\varepsilon_p$ an {\em Erd\H{o}s--R\'enyi process} with parameter $p$.

%The Erd\H{o}s--R\'enyi process with $0<p<1$ is a special family of probability measures on $\graphsN$.
By including each edge independently with probability $p$, the finite-dimensional distributions of $\Gamma\sim\varepsilon_p$ satisfy
\begin{equation}\label{eq:ER-fidi}
\mathbb{P}\{\Gamma|_{[n]}=F\}=\varepsilon_p^{(n)}(F):=\prod_{1\leq i<j\leq n}p^{F^{ij}}(1-p)^{1-F^{ij}}>0,\quad F\in\graphsn,\quad n\in\Nb.\end{equation}
Thus, $\varepsilon_p$ has full support in the sense that it assigns positive measure to all open subsets of $\graphsN$.

Important to our proof of Theorems \ref{thm:discrete char}, \ref{thm:Levy-Ito}, and \ref{thm:discrete-rep}, $\Gamma\sim\varepsilon_p$ admits an Aldous--Hoover representation $\Gamma\equalinlaw(g_0(\zeta_{\emptyset},\zeta_{\{i\}},\zeta_{\{j\}},\zeta_{\{i,j\}}))_{i,j\geq1}$, where
\begin{equation}\label{eq:nice rep}
g_0(\zeta_{\emptyset},\zeta_{\{i\}},\zeta_{\{j\}},\zeta_{\{i,j\}})=g'_0(\zeta_{\{i,j\}}):=\left\{\begin{array}{cc}
1,& 0\leq \zeta_{\{i,j\}}\leq p,\\
0,& \text{otherwise.}
\end{array}\right.\end{equation}

%Even stronger than the above properties, the Erd\H{o}s--R\'enyi process produces a countable graph that is almost surely
%\begin{itemize}
%	\item[(U1)] {\em universal}, i.e., for every $m\in\mathbb{N}$ and $F\in\graphsm$, there exists an injection $\phi:[m]\rightarrow\mathbb{N}$ such that $G^{\phi}=F$, and 
%	\item[(U2)] {\em ultrahomogeneous}, i.e., for every $F\in\graphsn$ and injection $\phi:[m]\rightarrow\mathbb{N}$ for which $G^{\phi}=F|_{[m]}$, $\phi$ extends to an injection $\phi':[n]\rightarrow\Nb$ such that $G^{\phi'}=F$.
%\end{itemize}
%The {\em Rado graph} $\mathscr{R}$ is the unique (up to isomorphism) countable universal and ultrahomogeneous graph \cite{Rado1964} and, thus, $\varepsilon_{p}$-almost every $G\in\graphsN$ is isomorphic to $\mathscr{R}$.
%The importance of these properties will become apparent in Section \ref{section:one process} when we use (U1) and (U2) to embed any $\graphsN$-valued exchangeable Feller process $\Gbf$ into a single family $(\Gamma_t)_{t\in T}$.
%This aspect of Erd\H{o}s--R\'enyi graphs is also central to the description of exchangeable Feller processes $\Gbf$ in Theorems \ref{thm:discrete char}, \ref{thm:Poisson}, and \ref{thm:Levy-Ito}.

\subsection{Feller processes}\label{section:graph-valued processes}

Let $\mathbf{Z}:=\{\mathbf{Z}_z:\,z\in\mathcal{Z}\}$ be a Markov process in any Polish space $\mathcal{Z}$.
The {\em semigroup} $(\mathbf{P}_t)_{t\in T}$ of $\mathbf{Z}$ acts on bounded, measurable functions $g:\mathcal{Z}\rightarrow\mathbb{R}$ by
\[\mathbf{P}_tg(z):=\mathbb{E}(g(Z_t)\mid \Gamma_0=z),\quad z\in\mathcal{Z}.\]
We say that $\mathbf{Z}$ possesses the {\em Feller property} if for all bounded, continuous $g:\mathcal{Z}\rightarrow\mathbb{R}$
\begin{itemize}
	\item[(i)] $z\mapsto\mathbf{P}_tg(z)$ is continuous for every $t>0$ and
	\item[(ii)] $\lim_{t\downarrow0}\mathbf{P}_tg(z)=g(z)$ for all $z\in\mathcal{Z}$.
\end{itemize}

\begin{prop}\label{prop:Feller equiv}
The following are equivalent for any exchangeable Markov process $\Gbf=\{\Gbf_{G}:\,G\in\graphsN\}$ on $\graphsN$.
\begin{itemize}
	\item[(i)] $\mathbf{\Gamma}$ is consistent.
	\item[(ii)] $\mathbf{\Gamma}$ has the Feller property.
\end{itemize}
\end{prop}
\begin{proof}

\noindent{\bf (i)$\Rightarrow$(ii)}: Compactness of the topology induced by \eqref{eq:metric} and the Stone--Weierstrass theorem imply that 
\[\{g:\graphsN\rightarrow\mathbb{R}:\,\exists n\in\Nb\text{ such that }G|_{[n]}=G'|_{[n]}\text{ implies }g(G)=g(G')\}\]
is dense in the space of continuous functions $\graphsN\rightarrow\mathbb{R}$.
Therefore, every bounded, continuous $g:\graphsN\rightarrow\Nb$ can be approximated to arbitrary accuracy by a function that depends on $G\in\graphsN$ only through $G|_{[n]}$, for some $n\in\Nb$.
For any such approximation, the conditional expectation in the definition of the Feller property depends only on the restriction of $\mathbf{\Gamma}$ to $\graphsn$, which is a Markov chain.
The first point in the Feller property follows immediately and the second follows soon after by noting that $\graphsn$ is a finite state space and, therefore, $\mathbf{\Gamma}^{[n]}$ must have a strictly positive hold time in its initial state with probability one.

\vspace{2mm}

\noindent{\bf (ii)$\Rightarrow$(i)}: 
For fixed $n\in\mathbb{N}$ and $F\in\graphsn$, we define $\psi_{F}:\graphsN\rightarrow\{0,1\}$ by
\[\psi_{F}(G):=\mathbf{1}\{G|_{[n]}=F\},\quad G\in\graphsN,\]
which is bounded and continuous on $\graphsN$.
To establish consistency, we must prove that
\[\mathbb{P}\{\Gamma_t|_{[n]}=F\mid \Gamma_0=G\}=\mathbb{P}\{\Gamma_t|_{[n]}=F\mid \Gamma_0|_{[n]}=G|_{[n]}\}\]
for every $t\geq0$, $G\in\graphsN$, and $F\in\graphsn$, for all $n\in\Nb$, which amounts to showing that 
\begin{equation}\label{eq:equality}
\mathbf{P}_t\psi_{F}(G)=\mathbf{P}_t\psi_{F}(G^*)\end{equation}
 for all $G,G^*\in\graphsN$ for which $G|_{[n]}=G^*|_{[n]}$.

By part (i) of the Feller property, $\mathbf{P}_t\psi_{F}(\cdot)$ is continuous and, therefore, it is enough to establish \eqref{eq:equality} on a dense subset of $\{G\in\graphsN:\,G|_{[n]}=F\}$.
Exchangeability implies that $\mathbf{P}_t\psi_{F}(G)=\mathbf{P}_t\psi_{F}(G^{\sigma})$ for all $\sigma\in\symmetricN$ that coincide with the identity $[n]\rightarrow[n]$.
To establish \eqref{eq:equality}, we identify $G^{*}\in\graphsN$ such that $G^*|_{[n]}=F$ and $\{G^{*\sigma}:\,\sigma\text{ coincides with the identity }[n]\rightarrow[n]\}$ is dense in $\{G\in\graphsN:\,G|_{[n]}=F\}$.

%Let $\mathscr{R}$ denote the {\em Rado graph} from Section \ref{section:ER}.
%By universality, there exists an injection $\phi:[n]\rightarrow\Nb$ such that $\mathscr{R}^{\phi}=F$.
%By ultrahomogeneity, there exists an extension $\bar{\phi}:\Nb\rightarrow\Nb$ of $\phi$ such that $\mathcal{R}^{\bar{\phi}}$ is isomorphic to $\mathscr{R}$.
%We define $G^{*}:=\mathscr{R}^{\bar{\phi}}$.
For this, we draw $G^*\sim\varepsilon_{1/2}$ conditional on the event $\{G^*|_{[n]}=F\}$.
Now, for every $F'\in\mathcal{G}_{[n']}$, $n'\geq n$, such that $F'|_{[n]}=F$, Kolmogorov's zero-one law guarantees a permutation $\sigma:\Nb\rightarrow\Nb$ that fixes $[n]$ and for which $G^{*\sigma}|_{[n']}=F'$.
Therefore, 
\[\{G^{*\sigma}:\,\sigma\in\symmetricN,\,\sigma\text{ coincides with identity }[n]\rightarrow[n]\}\]
is dense in $\{G\in\graphsN:\,G|_{[n]}=F\}$.
By the invariance of $\mathbf{P}_t\psi_{F}(\cdot)$ with respect to permutations that fix $[n]$, $\mathbf{P}_t\psi_{F}$ is constant on a dense subset of $\{G\in\graphsN:\,G|_{[n]}=F\}$.
Continuity implies that $\mathbf{P}_t\psi_{F}(\cdot)$ must be constant on the closure of this set and, therefore, can depend only on the restriction to $\graphsn$.

By part (ii) of the Feller property, 
\[\mathbf{P}_t\psi_{F}(G^*)\rightarrow\psi_{F}(G^*)\quad\text{as }t\downarrow0.\]
It follows that each finite restriction $\mathbf{\Gamma}^{[n]}$ has c\`adl\`ag paths and is consistent.
\end{proof}

\subsection{More on graph limits}\label{section:more graph limits}

Recall Definition \ref{defn:dissociated graph} of dissociated random graphs.
In this context, $\limitdensities$ corresponds to the space of exchangeable, dissociated probability measures on $\graphsN$, so we often conflate the two notions and write
\[D(F):=\gamma_D(\{G\in\graphsN:\,G|_{[n]}=F\}),\quad F\in\graphsn,\quad n\in\Nb,\]
for each $D\in\limitdensities$.
The space of graph limits $\limitdensities$ corresponds to the set of exchangeable, dissociated probability measures on $\graphsN$ and is compact under metric
\begin{equation}\label{eq:TV}
\xnorm{D-D'}:=\sum_{n\in\Nb}2^{-n}\sum_{F\in\graphsn}|D(F)-D'(F)|,\quad D,D'\in\limitdensities.
\end{equation}
%where $F_1,F_2,\ldots$ is any fixed enumeration of $\mathcal{G}^*$.
We furnish $\limitdensities$ with the trace of the Borel $\sigma$-field on $[0,1]^{\bigcup_{m\in\Nb}\graphsm}$.

The Aldous--Hoover theorem (Theorem \ref{thm:A-H}) provides the link between dissociated measures and exchangeable random graphs.
In particular, dissociated graphs are {ergodic} with respect to the action of $\symmetricN$ in the space of exchangeable random graphs, i.e., the law of every exchangeable random graph is a mixture of dissociated random graphs, and to every exchangeable random graph $\Gamma$ there is a unique probability measure $\Delta$ so that $\Gamma\sim\gamma_{\Delta}$ as in \eqref{eq:Delta-mixture}.

\begin{example}[Graph limit of the Erd\H{o}s--R\'enyi process]\label{ex:ER-graph limit}
For fixed $p\in[0,1]$, let $(\varepsilon_p^{(n)})_{n\in\mathbb{N}}$ be the collection of finite-dimensional Erd\H{o}s--R\'enyi measures in \eqref{eq:ER-fidi}.
The graph limit $|\Gamma|$ of $\Gamma\sim\varepsilon_p$ is deterministic  and satisfies $\delta(F,\Gamma)=\varepsilon_p^{(n)}$ a.s.\ for every $F\in\graphsn$, $n\in\Nb$.
\end{example}

\begin{example}[Random graph limit]\label{ex:random graph limit}
In general, the graph limit of an exchangeable random graph $\Gamma$ is a random element of $\limitdensities$.  
For example, let $\varepsilon_{\alpha,\beta}$ be the mixture of Erd\H{o}s--R\'enyi measures by the Beta distribution with parameter $(\alpha,\beta)$, i.e.,
\[\varepsilon_{\alpha,\beta}(\cdot):=\int_{[0,1]}\varepsilon_p(\cdot)\mathscr{B}_{\alpha,\beta}(dp),\]
where $\mathscr{B}_{\alpha,\beta}$ denotes the Beta distribution with parameter $(\alpha,\beta)$.
  The finite-dimensional distributions of $\Gamma\sim\varepsilon_{\alpha,\beta}$ are
\begin{equation}\label{eq:ER-mixture}
\mathbb{P}\{\Gamma|_{[n]}=F\}=\varepsilon_{\alpha,\beta}^{(n)}(F):=\frac{\alpha^{\uparrow n_1}\beta^{\uparrow n_0}}{(\alpha+\beta)^{\uparrow {n\choose 2}}},\quad F\in\graphsn,\quad n\in\Nb,\end{equation}
where $n_1:=\sum_{1\leq i<j\leq n}\mathbf{1}\{ij\in F\}$, $n_0:=\sum_{1\leq i<j\leq n}\mathbf{1}\{ij\notin F\}$, and $\alpha^{\uparrow j}:=\alpha(\alpha+1)\cdots(\alpha+j-1)$.  In this case, $\Gamma\sim\varepsilon_{\alpha,\beta}$ projects to a random element $|\Gamma|$ of $\limitdensities$, where $|\Gamma|$ is distributed over the graph limits of $\varepsilon_p$-distributed random graphs for $p$ following the Beta distribution with parameter $(\alpha,\beta)$.
\end{example}

\begin{rmk}[Connection to Lov\'asz--Szegedy graph limit theory]
Lov\'asz and Szegedy \cite{LovaszSzegedy2006} define a graph limit in terms of a measurable function $\phi:[0,1]\times[0,1]\rightarrow[0,1]$, called a {\em graphon}.
To see the connection to our definition as an exchangeable, dissociated probability measure, we can generate a random graph $\Gamma$ by first taking $U_1,U_2,\ldots$ i.i.d.\ Uniform random variables on $[0,1]$.  Conditional on $U_1,U_2,\ldots$, we generate each $\Gamma^{ij}$ independently according to
\[\mathbb{P}\{\Gamma^{ij}=1\mid U_1,U_2,\ldots\}=\phi(U_i,U_j),\quad i<j.\]
The distribution of $\Gamma$ corresponds to $\gamma_D$ for some $D\in\limitdensities$.

Some comments about the Lov\'asz--Szegedy notion of {graph limit}.
\begin{itemize}
	\item[(1)]  For a given sequence of graphs $(G_n)_{n\in\Nb}$, Lov\'asz and Szegedy's graphon (\cite{LovaszSzegedy2006}, Section 1) is not unique.
%; thus, their construal of a graphon as a {\em limit} is a misnomer. 
In contrast, our interpretation of a graph limit as a probability measure is essentially unique (up to sets of measure zero).
	\item[(2)] Lov\'asz and Szegedy's definition of graphon is a  recasting of the Aldous--Hoover theorem in the special case of countable graphs; see Theorem \ref{thm:A-H} above.
	\item[(3)] Our definition of a countable graph $G\in\graphsN$ as the {\em projective limit} of a compatible sequence of finite graphs $(G_n)_{n\in\Nb}$ should not be confused with a {\em graph limit}.  Any $G\in\graphsN$ corresponds to a unique compatible sequence $(G_1,G_2,\ldots)$ of finite graphs, whereas a graph limit $D\in\limitdensities$ corresponds to the measurable set of countable graphs $|D|^{-1}:=\{G\in\graphsN:|G|=D\}$. 
\end{itemize}
\end{rmk}

\subsection{Rewiring maps and their limits}\label{section:wiring maps}

As defined in Section \ref{section:discrete-intro}, a {rewiring map} $W$ is a symmetric $\{0,1\}\times\{0,1\}$-valued array $(W^{ij})_{i,j\geq1}$ with $(0,0)$ along the diagonal.
We often regard $W$ as a map $\graphsN\rightarrow\graphsN$, a pair $(W_0,W_1)\in\graphsN\times\graphsN$, or a  $\{0,1\}\times\{0,1\}$-valued array, without explicit specification.

Given any $W\in\wiren$ and any injection $\phi:[m]\rightarrow[n]$, we define $W^{\phi}:=(W^{\phi(i)\phi(j)})_{1\leq i,j\leq m}$.  
In particular,  $W|_{[n]}:=(W_0|_{[n]},W_1|_{[n]})$ denotes the restriction of $W$ to a rewiring map $\graphsn\rightarrow\graphsn$ and, for any $\sigma\in\symmetricN$, $W^{\sigma}$ is the image of $W$ under relabeling by $\sigma$.
We equip $\wireN$ with the product-discrete topology induced by
\[d(W,W'):=1/\max\{n\in\mathbb{N}:W|_{[n]}=W'|_{[n]}\},\quad W,W'\in\wireN,\]
and the associated Borel $\sigma$-field $\sigma\langle\,\cdot|_{[n]}\rangle_{n\in\Nb}$ generated by restriction maps.

\begin{defn}[Exchangeable rewiring maps]
A random rewiring map $W\in\wireN$ is {\em exchangeable} if $W\equalinlaw W^{\sigma}$ for all $\sigma\in\symmetricN$.
A measure $\omega$ on $\wireN$ is {\em exchangeable} if 
\[\omega(S)=\omega(S^{\sigma})\quad\text{for all }\sigma\in\symmetricN,\]
for every measurable subset $S\subseteq\wireN$, where $S^{\sigma}:=\{W^{\sigma}:W\in S\}$.  In particular, a probability measure $\omega$ on $\wireN$ is exchangeable if it determines the law of an exchangeable random rewiring map.
\end{defn}

Appealing to the notion of graph limits, we define the {\em density of $V\in\mathcal{W}_{[m]}$ in $W\in\wiren$}  by
\[\delta(V,W):={\ind(V,W)}/{n^{\downarrow m}},\]
where, as in \eqref{eq:induced density}, $\ind(V,W)$ is the number of injections $\phi:[m]\rightarrow[n]$ for which $W^{\phi}=V$.  
We define the {\em limiting density of $V\in\mathcal{W}_{[m]}$ in $W\in\wireN$} by
\[\delta(V,W):=\lim_{n\rightarrow\infty}\delta(V,W|_{[n]}),\quad\text{if it exists}.\]

\begin{defn}[Rewiring limits]
The {\em rewiring limit} of $W\in\wireN$ is the collection 
\[|W|:=(\delta(V,W))_{V\in\bigcup_{m\in\Nb}\mathcal{W}_{[m]}},\]
provided $\delta(V,W)$ exists for every $V\in\bigcup_{m\in\Nb}\mathcal{W}_{[m]}$.
If $\delta(V,W)$ does not exist for some $V$, then we put $|W|:=\partial$.
We write $\limitwiredensities$ to denote the closure of $\{|W|:\,W\in\wireN\}\setminus\{\partial\}$ in $[0,1]^{\bigcup_{m\in\Nb}\wirem}$.
\end{defn}

Following the program of Section \ref{section:graph limits}, every $\upsilon\in\limitwiredensities$ determines an exchangeable probability measure $\Omega_{\upsilon}$ on $\wireN$.
We call $\Omega_{\upsilon}$ the {\em rewiring measure directed by $\upsilon$}, which is the essentially unique exchangeable measure on $\wireN$ such that $\Omega_{\upsilon}$-almost every $W\in\wireN$ has $|W|=\upsilon$.  
Given any measure $\Upsilon$ on $\limitwiredensities$, we define the mixture of rewiring measures by
\begin{equation}\label{eq:rewiring measure}
\Omega_{\Upsilon}(\cdot):=\int_{\limitwiredensities}\Omega_{\upsilon}(\cdot)\Upsilon(d\upsilon).\end{equation}

The space of rewiring limits $\limitwiredensities$ corresponds to the closure of exchangeable, dissociated probability measures on $\wireN$ and, therefore, is compact under the analogous metric to \eqref{eq:TV}, 
\[\xnorm{\upsilon-\upsilon'}:=\sum_{n\in\Nb}2^{-n}\sum_{V\in\wiren}|\upsilon(V)-\upsilon'(V)|,\quad\upsilon,\upsilon'\in\limitwiredensities.\]
%for any fixed enumeration $V_1,V_2,\ldots$ of $\mathcal{W}^*$.
As for the space of graph limits,  we furnish $\limitwiredensities$ with the trace Borel $\sigma$-field of $[0,1]^{\bigcup_{m\in\Nb}\mathcal{W}_{[m]}}$.

\begin{lemma}\label{lemma:existence limit density}
The following hold for general rewiring maps.
\begin{itemize}
	\item[(i)]  As a map $\graphsN\rightarrow\graphsN$, every $W\in\wireN$ is Lipschitz continuous with Lipschitz constant 1 and the rewiring operation is associative in the sense that
\[(W\circ W')\circ W''=W\circ(W'\circ W''),\quad\text{for all } W,W',W''\in\wireN.\]
	\item[(ii)] $(W(G))^{\sigma}=W^{\sigma}(G^{\sigma})$ for all $W\in\wireN$, $G\in\graphsN$, and $\sigma\in\symmetricN$.
	\item[(iii)] If $W\in\wireN$ is an exchangeable rewiring map, then $|W|$ exists almost surely.  Moreover, for every $\upsilon\in\limitwiredensities$, $\Omega_{\upsilon}$-almost every $W\in\wireN$ has $|W|=\upsilon$.
\end{itemize}
\end{lemma}
\begin{proof}
Parts (i) and (ii) are immediate from the definition of rewiring.
Part (iii) follows from a straightforward modification of the almost sure existence of $|\Gamma|$ for any exchangeable random graph $\Gamma$, cf.\ Theorem \ref{thm:A-H}.
\end{proof}

\section{Discrete-time chains and the rewiring measure}\label{section:discrete}

\subsection{Conditional Aldous--Hoover theorem for graphs}

\begin{thm}[Conditional Aldous--Hoover theorem for graphs]\label{thm:discrete-rep}
 Let $\mathbf{\Gamma}$ be an exchangeable Feller chain on $\graphsN$ whose transitions are governed by transition probability measure $P$.  
Then there exists a measurable function $f:[0,1]^4\times\{0,1\}\rightarrow\{0,1\}$ satisfying $f(\cdot, b,c,\cdot,\cdot)=f(\cdot,c,b,\cdot,\cdot)$ such that, for every fixed $G\in\graphsN$, the distribution  of $\Gamma'\sim P(G,\cdot)$ is identical to that of $\Gamma'^*:=(\Gamma'^{*ij})_{i,j\geq1}$, where
\begin{equation}\label{eq:undirected WE rep}\Gamma'^{*ij}=f(\zeta_{\emptyset},\zeta_{\{i\}},\zeta_{\{j\}},\zeta_{\{i,j\}},G^{ij}),\quad i,j\geq1,\end{equation}
for $\{\zeta_{\emptyset};(\zeta_{\{i\}})_{i\geq1};(\zeta_{\{i,j\}})_{j>i\geq1}\}$ i.i.d.\ Uniform random variables on $[0,1]$.
\end{thm}

The representation in \eqref{eq:undirected WE rep} requires a delicate argument.
By the implication (ii)$\Rightarrow$(i) in Proposition \ref{prop:Feller equiv}, we can immediately deduce a weaker representation for $\Gamma'$ by $\Gamma'^*=(\Gamma'^{*ij})_{i,j\geq1}$ with
\[\Gamma'^{*ij}=f(\zeta_{\emptyset},\zeta_{\{i\}},\zeta_{\{j\}},\zeta_{\{i,j\}},G|_{[i\vee j]}),\quad i,j\geq1,\]
where $i\vee j:=\max(i,j)$.
The stronger representation in \eqref{eq:undirected WE rep} says that the conditional distribution of each entry of $\Gamma'^{*}$ depends only on the corresponding entry of $G$.

\begin{proof}[Proof of Theorem \ref{thm:discrete-rep}]
Let $\Gbf=\{\Gbf_G:\,G\in\graphsN\}$ be an exchangeable Feller chain on $\graphsN$ with transition probability measure $P$.
Let $\varepsilon_{1/2}$ denote the Erd\H{o}s--R\'enyi measure with $p=1/2$.
To prove \eqref{eq:undirected WE rep}, we generate a jointly exchangeable pair $(\Gamma,\Gamma')$, where 
\[\Gamma\sim\varepsilon_{1/2}\quad\text{and}\quad\Gamma'\mid\Gamma=G\sim P(G,\cdot).\]
By exchangeability of $\varepsilon_{1/2}$ and the transition probability measure $P$, $(\Gamma,\Gamma')$ is jointly exchangeable, i.e., $(\Gamma^{\sigma},\Gamma'^{\sigma})\equalinlaw(\Gamma,\Gamma')$ for all $\sigma\in\symmetricN$, and, therefore, determines a weakly exchangeable $\{0,1\}\times\{0,1\}$-valued array.
By the Aldous--Hoover theorem, there exists a measurable function $f:[0,1]^4\rightarrow\{0,1\}\times\{0,1\}$ such that $(\Gamma,\Gamma')\equalinlaw(\Gamma^*,\Gamma'^*)$, where
\[(\Gamma^*,\Gamma'^*)^{ij}=f(\zeta_{\emptyset},\zeta_{\{i\}},\zeta_{\{j\}},\zeta_{\{i,j\}}),\quad i,j\geq1,\]
for $\{\zeta_{\emptyset};(\zeta_{\{i\}})_{i\geq1};(\zeta_{\{i,j\}})_{j>i\geq1}\}$ i.i.d.\ Uniform$[0,1]$.
By separating components, we can write $f=(f_0,f_1)$ so that
\[(\Gamma^{*ij},\Gamma'^{*ij})=(f_0(\zeta_{\emptyset},\zeta_{\{i\}},\zeta_{\{j\}},\zeta_{\{i,j\}}),f_1(\zeta_{\emptyset},\zeta_{\{i\}},\zeta_{\{j\}},\zeta_{\{i,j\}})),\quad i,j\geq1.\]
With $g_0$ and $g_0'$ defined as in \eqref{eq:nice rep}, we have 
\[(g_0(\zeta_{\emptyset},\zeta_{\{i\}},\zeta_{\{j\}},\zeta_{\{i,j\}}))_{i,j\geq1}\sim\varepsilon_{1/2}\]
 and
\[(f_0(\zeta_{\emptyset},\zeta_{\{i\}},\zeta_{\{j\}},\zeta_{\{i,j\}}))_{i,j\geq1}\equalinlaw(g_0(\zeta_{\emptyset},\zeta_{\{i\}},\zeta_{\{j\}},\zeta_{\{i,j\}}))_{i,j\geq1}=(g_0'(\zeta_{\{i,j\}})_{i,j\geq1}.\]

Using Kallenberg's notation \cite{KallenbergSymmetries}, we write $\hat{\zeta}_J=(\zeta_I)_{I\subseteq J}$ for subsets $J\subset\Nb$, so that $\hat{\zeta}_{\{i,j\}}=(\zeta_{\emptyset},\zeta_{\{i\}},\zeta_{\{j\}},\zeta_{\{i,j\}})$, $i,j\geq1$, and $\hat{\zeta}_{\{i\}}=(\zeta_{\emptyset},\zeta_{\{i\}})$, $i\geq1$.
By Theorem 7.28 of \cite{KallenbergSymmetries}, there exists a measurable function $h:[0,1]^2\cup[0,1]^4\cup[0,1]^8\rightarrow[0,1]$ such that for $\{\eta_{\emptyset};(\eta_{\{i\}})_{i\geq1};(\eta_{\{i,j\}})_{j>i\geq1}\}$ i.i.d.\ Uniform$[0,1]$ random variables, $h(\hat{\zeta}_J,\hat{\eta}_J)\sim\text{Uniform}[0,1]$ and is independent of $\hat{\zeta}_J\setminus\{\zeta_J\}$ and $\hat{\eta}_J\setminus\{\eta_J\}$ for every $J\subseteq\Nb$ with two or fewer elements and
\[g_0'(\zeta_{\{i,j\}})=g_0(\hat{\zeta}_{\{i,j\}})=f_0(h(\zeta_{\emptyset},\eta_{\emptyset}),h(\hat{\zeta}_{\{i\}},\hat{\eta}_{\{i\}}),h(\hat{\zeta}_{\{j\}},\hat{\eta}_{\{j\}}),{h}(\hat{\zeta}_{\{i,j\}},\hat{\eta}_{\{i,j\}}))\quad\text{a.s.\ for every }i,j\geq1.\]

Now, put $\xi_{\emptyset}=h(\zeta_{\emptyset},\eta_{\emptyset})$, $\xi_{\{i\}}=h(\hat{\zeta}_{i},\hat{\eta}_{\{i\}})$, and $\xi_{\{i,j\}}=h(\hat{\zeta}_{ij},\hat{\eta}_{\{i,j\}})$ so that $\{\xi_{\emptyset};(\xi_{\{i\}})_{i\geq1};(\xi_{\{i,j\}})_{j>i\geq1}\}$ are i.i.d.\ Uniform$[0,1]$ random variables.
The generic Aldous--Hoover representation allows us to construct $(\Gamma^{*},\Gamma'^{*})$ by
\[(\Gamma^{*ij},\Gamma'^{*ij})=(f_0(\hat{\xi}_{\{i,j\}}),f_1(\hat{\xi}_{\{i,j\}})),\quad i,j\geq1.\]
From Kallenberg's Theorem 7.28, $(\Gamma^{*},\Gamma'^{*})$ also has the representation
\[(\Gamma^{*ij},\Gamma'^{*ij})=(f_0(\hat{\xi}_{\{i,j\}}),f_1(\hat{\xi}_{\{i,j\}}))=(g_0(\hat{\zeta}_{\{i,j\}}),g_1(\hat{\zeta}_{\{i,j\}},\hat{\eta}_{\{i,j\}})),\quad\text{a.s.},\quad i,j\geq1,\]
where $g_1:[0,1]^8\rightarrow\{0,1\}$ is the composition of $f_1$ with $h$.
Again by the Coding Lemma, we can represent $(\hat{\zeta}_{\{i,j\}},\hat{\eta}_{\{i,j\}})_{i,j\geq1}$ by
\[(\hat{\zeta}_{\{i,j\}},\hat{\eta}_{\{i,j\}})_{i,j\geq1}\equalinlaw (a(\hat{\alpha}_{\{i,j\}}))_{i,j\geq1},\]
where $\{\alpha_{\emptyset};(\alpha_{\{i\}})_{i\geq1};(\alpha_{\{i,j\}})_{j>i\geq1}\}$ are i.i.d.\ Uniform$[0,1]$.

It follows that $(\Gamma,\Gamma')$ possesses a joint representation by 
\begin{align}
\Gamma^{*ij}&=g_0'(\alpha_{\{i,j\}})\quad\text{and}\label{eq:iid}\\
\Gamma'^{*ij}&=g_1'(\alpha_{\emptyset},\alpha_{\{i\}},\alpha_{\{j\}},\alpha_{\{i,j\}}).\notag
\end{align}
%where 
%\[g_1(\zeta_{\emptyset},\zeta_{\{i\}},\zeta_{\{j\}},\zeta_{\{i,j\}})=f_1(\zeta_{\emptyset},\zeta_{\{i\}},\zeta_{\{j\}},h(\hat{\zeta}_{\{i,j\}},\hat{\eta}_{\{i,j\}})).\]
By the Coding Lemma, we can represent
\[(g_0'(\alpha_{\{i,j\}}),\alpha_{\{i,j\}})\equalinlaw(g_0'(\alpha_{\{i,j\}}),u(g_0'(\alpha_{\{i,j\}}),\chi_{\{i,j\}}))\]
for $(\chi_{\{i,j\}})_{j>i\geq1}$ i.i.d.\ Uniform$[0,1]$ random variables independent of $(\alpha_{\{i,j\}})_{j>i\geq1}$.
%For fixed $G\in\graphsN$, we can represent $\zeta_{\{i,j\}}=u(G^{ij},\chi_{\{i,j\}})$, for $\chi_{\{i,j\}}$ i.i.d.\ Uniform$[0,1]$ random variables independent of $(\zeta_{\{i,j\}})_{j>i\geq1}$.
Thus, we define
\[g'_1(\alpha_{\emptyset},\alpha_{\{i\}},\alpha_{\{j\}},\chi_{\{i,j\}},G^{ij}):=g_1(\alpha_{\emptyset},\alpha_{\{i\}},\alpha_{\{j\}},u(G^{ij},\chi_{\{i,j\}}))\equalinlaw g_1(\alpha_{\emptyset},\alpha_{\{i\}},\alpha_{\{j\}},\alpha_{\{i,j\}}),\]
so that
\[(\Gamma,\Gamma')\equalinlaw(g_0'({\alpha}_{\{i,j\}}),g'_1(\alpha_{\emptyset},\alpha_{\{i\}},\alpha_{\{j\}},\alpha_{\{i,j\}},g'_0(\alpha_{\{i,j\}})))_{i,j\geq1}.\]
Conditioning on $\{\Gamma=G\}$ gives
\[P(G,\cdot)=\mathbb{P}\{(g'_1(\alpha_{\emptyset},\alpha_{\{i\}},\alpha_{\{j\}},\alpha_{\{i,j\}},G^{ij}))_{i,j\geq1}\in\cdot\}\quad\text{for }\varepsilon_{1/2}\text{-almost every }G\in\graphsN.\]
By the Feller property, $P(G,\cdot)$ is continuous in the first argument and, thus, the above equality holds for all $G\in\graphsN$ and representation  \eqref{eq:undirected WE rep} follows.
\end{proof}

%\subsection{Construction from a single process}\label{section:one process}
%
%Our proof of Theorem \ref{thm:discrete-rep} uses only the fact that $\varepsilon_{1/2}$ has full support on $\graphsN$; however, we could say quite a bit more by using the stronger properties of universality and ultrahomogeneity of Erd\H{o}s--R\'enyi graphs; see Section \ref{section:ER}.
%These properties permit a simultaneous embedding of $\Gbf=\{\Gbf_G:\,G\in\graphsN\}$ within a single process $\Gbf^{\dagger}=(\Gamma_t^{\dagger})_{t\in T}$, which we generate from $\Gbf$ by
%\begin{itemize}
%	\item $\Gamma^{\dagger}_0\sim\varepsilon_{1/2}$ and
%	\item $\Gbf^{\dagger}=\Gbf_G$ on the event $\Gamma^{\dagger}_0=G$.
%\end{itemize}
%For each $G\in\graphsN$, we let $\phi_G:\Nb\rightarrow\Nb$ be the injection obtained by choosing $\phi_G(1)<\phi_G(2)<\cdots$ successively such that each $\phi_G(n)$ is the smallest integer such that $\Gamma_0^{\dagger\phi_G|_{[n]}}=G|_{[n]}$, where $\phi_G|_{[n]}$ is the domain restriction of $\phi_G$ to a map $[n]\rightarrow\Nb$.
%Thus $\Gamma_0^{\dagger\phi_G}=G$ for every $G\in\graphsN$.
%The existence of such $\phi_G$ for every $G\in\graphsN$ is a consequence of the ultrahomogeneity property of $\Gamma_0^{\dagger}$.
%Exchangeability and consistency properties of $\Gbf$ guarantee that $\Gbf^{\dagger\phi_G}\equalinlaw\Gbf_G$ for every $G\in\graphsN$.

\subsection{Discrete-time characterization}

The above construction of $\Gbf$ from the single process $\Gbf^{\dagger}$ is closely tied to the even stronger representation in Theorem \ref{thm:discrete char}, according to which $\Gbf$ can be constructed from a single i.i.d.\ sequence of exchangeable rewiring maps.

\begin{prop}
For an exchangeable probability measure $\omega$ on $\wireN$, let $\mathbf{\Gamma}^*_{\omega}=\{\Gbf^*_{G,\omega}:\,G\in\graphsN\}$ be constructed as in \eqref{eq:iterated construction} from an i.i.d.\ sequence $W_1,W_2,\ldots$ from $\omega$.  Then $\mathbf{\Gamma}^*_{\omega}$ is an exchangeable Feller chain on $\graphsN$.
\end{prop}

\begin{proof}
By our assumption that $W_1,W_2,\ldots$ are independent and identically distributed, each $\mathbf{\Gamma}^*_{G,\omega}$ has the Markov property.
Exchangeability of $\omega$ and Lemma \ref{lemma:existence limit density}(ii) pass exchangeability along to $\mathbf{\Gamma}^*_{\omega}$.
  By Lemma \ref{lemma:existence limit density}(i), every $W\in\wireN$ is Lipschitz continuous with Lipschitz constant 1, and so each $\mathbf{\Gamma}^{*[n]}_{G,\omega}$ also has the Markov property for every $n\in\mathbb{N}$.
By Proposition \ref{prop:Feller equiv}, $\mathbf{\Gamma}^*_{\omega}$ is Feller and the proof is complete.
\end{proof}

\begin{defn}[Rewiring Markov chains]
We call $\mathbf{\Gamma}^*_{\omega}$ constructed as in \eqref{eq:iterated construction} a {\em rewiring chain} with {\em rewiring measure} $\omega$.
If $\omega=\Omega_{\Upsilon}$ for some probability measure $\Upsilon$ on $\limitwiredensities$, we call $\mathbf{\Gamma}^*_{\Upsilon}$ a {\em rewiring chain directed by $\Upsilon$}.
\end{defn}

\begin{thm}\label{thm:rewiring char}
Let $\mathbf{\Gamma}$ be an exchangeable Feller chain on $\graphsN$.  Then there exists an exchangeable probability measure $\omega$ on $\wireN$ so that $\mathbf{\Gamma}\equalinlaw\mathbf{\Gamma}^*_{\omega}$, where $\mathbf{\Gamma}^*_{\omega}=\{\Gbf^*_{G,\omega}:\,G\in\graphsN\}$ is a rewiring Markov chain with rewiring measure $\omega$.
\end{thm}
\begin{proof}
By Theorem \ref{thm:discrete-rep}, there exists a measurable function $f:[0,1]^4\times\{0,1\}\rightarrow\{0,1\}$ so that the transitions of $\mathbf{\Gamma}$ can be constructed as in \eqref{eq:undirected WE rep}.  
From $f$, we define $f^*:[0,1]^4\rightarrow\{0,1\}^2$ by
\[f^*(a,b,c,d):=(f(a,b,c,d,0),f(a,b,c,d,1)).\]
Given  i.i.d.\ Uniform$[0,1]$ random variables $\{\zeta_{\emptyset};(\zeta_{\{i\}})_{i\geq1}; (\zeta_{\{i,j\}})_{j>i\geq1}\}$, we construct a weakly exchangeable $\{0,1\}^2$-valued array $W^*:=(W^{*ij})_{i,j\geq1}$ by
\[W^{*ij}:=f^*(\zeta_{\emptyset},\zeta_{\{i\}},\zeta_{\{j\}},\zeta_{\{i,j\}}),\quad i,j\geq1.\]
We write $\omega$ to denote the essentially unique distribution of $W^*$.

 Treating $W^*$ as a rewiring map, the image $W^*(G)=\Gamma^*$ satisfies
\[\Gamma^{*ij}=f(\zeta_{\emptyset},\zeta_{\{i\}},\zeta_{\{j\}},\zeta_{\{i,j\}},G^{ij})\quad\text{ for all }i,j\geq1,\]
for every $G\in\graphsN$; whence, $\mathbf{\Gamma}^*_{\omega}$ in \eqref{eq:iterated construction} has the correct transition probabilities.

The proof is complete.
\end{proof}

\subsection{Characterizing the rewiring measure}
Theorem \ref{thm:rewiring char} asserts that any exchangeable Feller chain on $\graphsN$ can be constructed as a rewiring chain.
In the discussion surrounding Lemma \ref{lemma:existence limit density} above, we identify every exchangeable probability measure $\omega$ with a unique probability measure $\Upsilon$ on $\limitwiredensities$ through the relation $\omega=\Omega_{\Upsilon}$, with $\Omega_{\Upsilon}$ defined in \eqref{eq:rewiring measure}.
Our next proposition records this fact.

\begin{prop}\label{prop:mixing measure discrete}
Let $\omega$ be an exchangeable probability measure on $\wireN$.  Then there exists an essentially unique probability measure $\Upsilon$ on $\limitwiredensities$ such that $\omega=\Omega_{\Upsilon}:=\int_{\limitwiredensities}\Omega_{\upsilon}\Upsilon(d\upsilon)$.
\end{prop}

\begin{proof}
This is a combination of the Aldous--Hoover theorem and Lemma \ref{lemma:existence limit density}(iii).  By Lemma \ref{lemma:existence limit density}(iii), $\omega$-almost every $W\in\wireN$ possesses a rewiring limit, from which the change of variables formula gives
\[\omega(dW):=\int_{\wireN}\Omega_{|V|}(dW)\omega(dV)=\int_{\limitwiredensities}\Omega_{\upsilon}(dW)|\omega|(d\upsilon)=\Omega_{|\omega|}(dW),\]
where $|\omega|$ denotes the image measure of $\omega$ by $|\cdot|:\wireN\rightarrow\limitwiredensities\cup\{\partial\}$.
\end{proof}

\begin{proof}[Proof of Theorem \ref{thm:discrete char}]
Follows from Theorem \ref{thm:rewiring char} and Proposition \ref{prop:mixing measure discrete}.
\end{proof}

\subsection{The induced chain on $\limitdensities$}\label{section:induced chain}
Any $\upsilon\in\limitwiredensities$ corresponds to a unique transition probability $P_{\upsilon}$ on $\graphsN\times\graphsN$, as defined in \eqref{eq:induced tps}.
Moreover, $\upsilon\in\limitwiredensities$ acts on $\limitdensities$ by composition of probability measures, as  in  \eqref{eq:right action}.  
For $\upsilon,\upsilon'\in\limitwiredensities$, we define the product $P'=P_{\upsilon}P_{\upsilon'}$ by
\begin{equation}\label{eq:product upsilon matrices}
%P'_{ij}:=\left\{\begin{array}{cc}
%\sum_{k:F_k\in\graphsn}P_{\upsilon}(i,k)P_{\upsilon'}(k,j),& F_i,F_j\in\graphsn\text{ for some }n\in\mathbb{N}\\
%0,&\text{otherwise.}\end{array}\right.
P'(G,dG'):=\int_{\graphsN}P_{\upsilon'}(G'',dG')P_{\upsilon}(G,dG''),\quad G,G'\in\graphsN.
\end{equation}

\begin{lemma}\label{lemma:existence wire density}
Let $W,W'\in\wireN$ be independent exchangeable random rewiring maps such that $|W|$ and $|W'|$ exist almost surely.  Then $|W'\circ W|$ exists almost surely and
\[P_{|W'\circ W|}=P_{|W|}P_{|W'|}\quad\text{a.s.}\]
\end{lemma}
\begin{proof}
Follows by independence of $W$ and $W'$, the definition of $W'\circ W$, and the definition of $P_{|W|}P_{|W'|}$ in \eqref{eq:product upsilon matrices}.
\end{proof}

\begin{proof}[Proof of Theorem \ref{thm:induced density chain}]
By Theorem \ref{thm:discrete char}, there exists a measure $\Upsilon$ on $\limitwiredensities$ such that we can regard $\Gbf$ as $\Gbf_{\Upsilon}^*$, the rewiring chain directed by $\Upsilon$.  
Thus, for every $G\in\graphsN$, $\Gbf_G=(\Gamma_m)_{m\geq0}$ is determined by $\Gamma_0=G$ and 
\[\Gamma_m\equalinlaw (W_m\circ\cdots\circ W_1)(G),\quad m\geq1,\]
where $W_1,W_2,\ldots$ are i.i.d.\ with distribution $\Omega_{\Upsilon}$.

For any $D\in\limitdensities$, $\Gbf_D=(\Gamma_m)_{m\geq0}$ is generated from $\Gbf=\{\Gbf_G:\,G\in\graphsN\}$ by first taking $\Gamma_0\sim\gamma_D$ and then putting $\Gbf_{D}=\Gbf_G$ on the event $\Gamma_0=G$.   
By the Aldous--Hoover theorem along with exchangeability of $\gamma_D$ and $\Omega_{\Upsilon}$, the graph limit $|\Gamma_m|$ exists almost surely for every $m\geq0$ and, thus, $|\mathbf{\Gamma}_{D}|:=(|\Gamma_m|)_{m\geq0}$ exists almost surely in $\limitdensities$.

Also, by Lemma \ref{lemma:existence limit density}(iii), $|W_1|,|W_2|,\ldots$ is i.i.d.\ from $\Upsilon$, and, by Lemma \ref{lemma:existence wire density},  $|W_m\circ\cdots\circ W_1|$ exists almost surely and 
\[P_{|W_m\circ\cdots\circ W_1|}=P_{|W_1|}\cdots P_{|W_m|}\quad\text{a.s.}\text{ for every }m\geq1.\]
Finally, by Theorem \ref{thm:discrete char}, 
\[|\Gamma_m|\equalinlaw|(W_m\circ\cdots\circ W_1)(\Gamma_0)|=|\Gamma_0|{|W_1|}\cdots {|W_m|}\quad\text{a.s.\ for every }m\geq1,\]
and so each $|\mathbf{\Gamma}_{D}|$ is a Markov chain on $\limitdensities$ with initial state $D$ and representation \eqref{eq:iterated product}.

To establish the Feller property, we use the fact that $\limitdensities$ is compact and, therefore, any continuous function $h:\limitdensities\rightarrow\mathbb{R}$ is uniformly continuous and bounded.
The dominated convergence theorem and Lipschitz continuity of the action of $\upsilon\in\limitwiredensities$ on $\limitdensities$ immediately give part (i) of the Feller property.  
Part (ii) of the Feller property is trivial because $|\Gbf_{\limitdensities}|$ is a discrete-time Markov chain.
\end{proof}

We could interpret $P'$  in \eqref{eq:product upsilon matrices} as the two-step transition probability measure for a time-inhomogeneous Markov chain on $\graphsN$ with first step governed by $P_{\upsilon}$ and second step governed by $P_{\upsilon'}$.
By this interpretation, Theorem \ref{thm:induced density chain} leads to another description of any exchangeable Feller chain $\mathbf{\Gamma}$ as a mixture of time-inhomogeneous Markov chains.
In particular, we can construct $\mathbf{\Gamma}$ by first sampling $Y_1,Y_2,\ldots$ i.i.d.\ from the unique directing measure $\Upsilon$.
Each $Y_i$ induces an exchangeable transition probability $P_{Y_i}$ as in \eqref{eq:induced tps}.
At each time $m\geq1$, given $\Gamma_{m-1},\ldots,\Gamma_0$ and $Y_1,Y_2,\ldots$, we generate $\Gamma_{m}\sim P_{Y_m}(\Gamma_{m-1},\cdot)$.
%The unconditional law of $\mathbf{\Gamma}=(\Gamma_m,\,m\geq0)$ with this construction is 
%This  immediately induces an i.i.d.\ sequence of random transition probability matrices $P^{(n)}_{Y_1},P^{(n)}_{Y_2},\ldots$ on $\graphsn$.  At each time %$m\geq1$, we obtain $\Gamma_{m}|_{[n]}$ given $\Gamma_{m-1}|_{[n]}=G$ by sampling from $P^{(n)}_{Y_m}(G,\cdot)$.  In this way, conditional on %$Y_1,Y_2,\ldots$, $\mathbf{\Gamma}^{[n]}$ is a time-inhomogeneous Markov chain on $\graphsn$.  By construction, the finite-dimensional transition %probability matrices $(P^{(n)}_{Y_m},\,n\in\mathbb{N})$ are consistent for every $m\geq1$, and so this procedure can be carried out simultaneously for %all $n\in\mathbb{N}$ to obtain a consistent system of Markov chains on $(\graphsn,\,n\in\mathbb{N})$.

\subsection{Examples}

\begin{example}[Product Erd\H{o}s--R\'enyi chains]\label{ex:ER-chain}
Let $(p_0,p_1)\in(0,1)\times(0,1)$ and put $\omega=\varepsilon_{p_0}\otimes\varepsilon_{p_1}$, where $\varepsilon_p$ denotes the Erd\H{o}s--R\'enyi measure with parameter $p$.  
This measure generates transitions out of state $G\in\graphsN$ by flipping a $p_k$-coin with $k=G^{ij}$ independently for every $j>i\geq1$.  The limiting trajectory in $\limitdensities$ is deterministic and settles down to degenerate stationary distribution at the graph limit of an $\varepsilon_q$-random graph, with $q:=p_0/(1-p_1+p_0)$.  
\end{example}

\begin{example}[A reversible family of Markov chains]\label{ex:reversible}

The model in Example \ref{ex:ER-chain} generalizes by mixing $(p_0,p_1)$ with respect to any measure on $[0,1]\times[0,1]$.  For example, let $\alpha,\beta\geq0$ and define $\omega=\varepsilon_{P_0}\otimes\varepsilon_{P_1}$ for $(P_0,P_1)\sim\mathscr{B}_{\alpha,\beta}\otimes\mathscr{B}_{\beta,\alpha}$, where $\mathscr{B}_{\alpha,\beta}$ is the Beta distribution with parameter $(\alpha,\beta)$.  
The finite-dimensional transition probabilities of this mixture model are
\[\mathbb{P}\{\Gamma_{m+1}^{[n]}=F'\mid\Gamma_m^{[n]}=F\}=\frac{\alpha^{\uparrow n_{00}}\beta^{\uparrow n_{01}}\beta^{\uparrow n_{10}}\alpha^{\uparrow n_{11}}}{(\alpha+\beta)^{\uparrow n_{0\bullet}}(\alpha+\beta)^{\uparrow n_{1\bullet}}},\quad F,F'\in\graphsn,\]
where $n_{rs}:=\sum_{1\leq i<j\leq n}\mathbf{1}\{F^{ij}=r,\,F'^{ij}=s\}$ and $n_{r\bullet}:=n_{r0}+n_{r1}$, $r=0,1$.  These transition probabilities are reversible with respect to
\[\mathbb{P}\{\Gamma^{[n]}=F\}=\frac{(\alpha+\beta)^{\uparrow n_0}(\alpha+\beta)^{\uparrow n_1}}{(2\alpha+2\beta)^{\uparrow {n\choose 2}}},\quad F\in\graphsn,\]
where $n_r:=\sum_{1\leq i<j\leq n}\mathbf{1}\{F^{ij}=r\}$, for $r=0,1$.  %See \cite{Crane2013Graphs0} for discussion of these and related chains.

\end{example}

\section{Continuous-time processes and the L\'evy--It\^o measure}\label{section:continuous}
We now let $\mathbf{\Gamma}:=\{\Gbf_G:\,G\in\graphsN\}$ be a continuous-time exchangeable Feller process on $\graphsN$.  
Any such process can jump infinitely often in arbitrarily small time intervals; however, by the consistency property \eqref{eq:consistent}, every finite restriction $\mathbf{\Gamma}^{[n]}$ is a finite-state space Markov chain and can jump only finitely often in bounded intervals. 
As we show, the interplay between these possibilities restricts the behavior of $\mathbf{\Gamma}$ in a precise way.

As before, $\id$ stands for the identity map $\graphsN\rightarrow\graphsN$.  Let $\omega$ be an exchangeable measure on $\wireN$ that satisfies
\begin{equation}\label{eq:regularity omega1}
\omega(\{\id\})=0\quad\text{and}\quad\omega(\{W\in\wireN:W|_{[2]}\neq\idtwo\})<\infty.
\end{equation}
We proceed as in \eqref{eq:PPP construction} and construct a process $\mathbf{\Gamma}^*_{\omega}:=\{\Gamma^*_{G,\omega}:\,G\in\graphsN\}$ on $\graphsN$ from a Poisson point process $\mathbf{W}=\{(t,W_t)\}\subseteq\mathbb{R}^+\times\wireN$ with intensity $dt\otimes\omega$.

\begin{prop}\label{prop:suff}
Let $\omega$ be an exchangeable measure satisfying \eqref{eq:regularity omega1}.
Then $\mathbf{\Gamma}^*_{\omega}$ constructed in \eqref{eq:PPP construction} is an exchangeable Feller process on $\graphsN$.
\end{prop}
\begin{proof}
For each $n\in\Nb$, $\mathbf{\Gamma}^{*[n]}_{\omega}$ is the restriction of $\mathbf{\Gamma}^*_{\omega}$ to a process on $\graphsn$.
By construction, $(\mathbf{\Gamma}^{*[n]}_{\omega})_{n\in\Nb}$ is a compatible collection of Markov processes and, thus, determines a process $\mathbf{\Gamma}^*_{\omega}$ on $\graphsN$.  By exchangeability of $\omega$ and \eqref{eq:regularity omega1},
\begin{eqnarray*}
\omega(\{W\in\wireN:\,W|_{[n]}\neq\idn\})&=&\omega\left(\bigcup_{1\leq i<j\leq n}\{W\in\wireN:\,W|_{\{i,j\}}\neq\mathbf{Id}_{\{i,j\}}\}\right)\\
&\leq&\sum_{1\leq i<j\leq n}\omega(\{W\in\wireN:\,W|_{\{i,j\}}\neq\mathbf{Id}_{\{i,j\}}\})\\
&=&\frac{n(n-1)}{2}\omega(\{W\in\wireN:\,W|_{[2]}\neq\idtwo\})\\
&<&\infty,
\end{eqnarray*}
so that $\omega$ satisfies \eqref{eq:regularity omega}.  
Also by exchangeability of $\omega$, each $\mathbf{\Gamma}^{*[n]}_{\omega}$ is an exchangeable Markov chain on $\graphsn$.
The limiting process $\mathbf{\Gamma}^*_{\omega}$ is an exchangeable Feller process on $\graphsN$ by Proposition \ref{prop:Feller equiv}.
\end{proof}

\begin{cor}
Every exchangeable measure $\omega$ satisfying \eqref{eq:regularity omega1} determines the jump rates of an exchangeable Feller process on $\graphsN$.
\end{cor}

By the Feller property, the transition law of each finite restriction $\Gbf^{[n]}$ is determined by the infinitesimal jump rates
\begin{equation}\label{eq:finite rates}
Q_n(F,F'):=\lim_{t\downarrow0}\frac{1}{t}\mathbb{P}\{\Gamma^{[n]}_t=F'\mid\Gamma^{[n]}_0=F\},\quad F,F'\in\graphsn,\quad F\neq F'.
\end{equation}
Exchangeability \eqref{eq:exch tps} and consistency \eqref{eq:consistent} of $\Gbf$ imply
\begin{align}\label{eq:exch rates}
Q_n(F^{\sigma},F'^{\sigma})&=Q_n(F,F')\quad\text{for all permutations }\sigma:[n]\rightarrow[n]\quad\text{and}\\
\label{eq:consistent rates}
Q_m(F|_{[m]},F')&=Q_n(F,\{F''\in\graphsn:\,F''|_{[m]}=F'\})=\sum_{F''\in\graphsn:\,F''|_{[m]}=F'}Q_n(F,F''),
\end{align}
 for all $F\in\graphsn$ and $F'\in\mathcal{G}_{[m]}$ such that $F|_{[m]}\neq F'$, $m\leq n$.
Thus, for every $G\in\graphsN$, $(Q_n)_{n\in\Nb}$ determines a pre-measure $Q(G,\cdot)$ on $\graphsN\setminus\{G\}$ by
\[Q(G,\{G'\in\graphsN:\,G'|_{[n]}=F'\}):=Q_n(G|_{[n]},F'),\quad F'\in\graphsn\setminus\{G|_{[n]}\},\quad n\in\Nb.\]
By \eqref{eq:consistent rates}, $Q(G,\cdot)$ is additive.
By Carath\'eodory's extension theorem, $Q(G,\cdot)$ has a unique extension to a measure on $\graphsN\setminus\{G\}$.

In $\mathbf{\Gamma}^*_{\omega}$ constructed above, $\omega$ determines the jump rates through
\[Q(G,\cdot)=\omega(\{W\in\wireN:W(G)\in \cdot\}),\quad G\in\graphsN\setminus\{G\}.\]
In fact, we can show that the infinitesimal rates of any exchangeable Feller process are determined by an exchangeable measure $\omega$ that satisfies \eqref{eq:regularity omega1}.  Our main theorem (Theorem \ref{thm:Levy-Ito}) gives a L\'evy--It\^o representation of this jump measure.  

\subsection{Existence of an exchangeable jump measure}\label{section:existence jump measure}
Let $(\mathbf{P}_t)_{t\geq0}$ be the Markov semigroup of an exchangeable Feller process $\mathbf{\Gamma}$ on $\graphsN$.
For every $t>0$, Theorems \ref{thm:discrete-rep} and \ref{thm:rewiring char} guarantee the existence of an exchangeable probability measure $\omega_t$ on $\wireN$ such that, for every $G\in\graphsN$,
\begin{equation}\label{eq:Omegat}\omega_t(\{W\in\wireN:W(G)\in\cdot\})=\mathbb{P}\{\Gamma_t\in\cdot\mid\Gamma_0=G\}.\end{equation}
By the relationship between $(\omega_t)_{t\geq0}$ and the time-homogeneous Markov process $(\Gamma_t)_{t\geq0}$, $\omega_t$ is exchangeable for every $t>0$, and the Chapman--Kolmogorov theorem implies that $(\omega_t)_{t>0}$ satisfies
\begin{equation}\label{eq:semigroup-omega}
\omega_{t+s}(\{W\in\wireN:\,W(G)\in\cdot\})
=\int_{\wireN}\omega_{t}(\{W'\in\wireN:(W'\circ W)(G)\in \cdot\})\omega_s(dW),
\end{equation}
for all $s,t\geq0$ and all $G\in\graphsN$.
The family $(\omega_t)_{t>0}$ is not determined by these conditions, but time-homogeneity and \eqref{eq:Omegat} allows us to choose a family $(\omega_t)_{t\geq0}$ that satisfies the further semigroup property
\[\omega_{t+s}(\cdot)=\int_{\wireN}\omega_t(\{W'\in\wireN:\,W'\circ W\in\cdot\})\omega_s(dW),\quad s,t\geq0.\]

By the Feller property, $W_t\sim\omega_t$ must also satisfy
\[\mathbf{P}_tg(G)=\mathbb{E}g(W_t(G))\rightarrow g(G)\quad\text{as }t\downarrow0\text{ for all }G\in\graphsN,\]
for all continuous $g:\graphsN\rightarrow\mathbb{R}$.  
Thus, as $t\downarrow0$, $W_t\rightarrow_{P}\id$ and $\omega_t\rightarrow_{w}\delta_{\{\id\}}$, the point mass at the identity map.  
(Here, $\rightarrow_P$ denotes convergence in probability and $\rightarrow_w$ denotes weak convergence of probability measures.)  By right-continuity at $t=0$, we have a family $(\omega_t)_{t\geq0}$ of measures on $\wireN$.

We obtain an infinitesimal jump measure $\omega$ on $\wireN\setminus\{\id\}$ through its  finite-dimensional restrictions $\omega^{(n)}$ on $\wiren\setminus\{\idn\}$:
\begin{equation}\label{eq:fidi jump rates}
\omega^{(n)}(V):=\lim_{t\downarrow0}\frac{1}{t}\omega_t(\{W\in\wireN:W|_{[n]}=V\}),\quad V\in\wiren\setminus\{\idn\},\quad n\in\Nb.\end{equation}

\begin{prop}\label{prop:existence omega}
The family of measures $(\omega^{(n)})_{n\in\Nb}$ in \eqref{eq:fidi jump rates} determines a unique exchangeable measure $\omega$ that satisfies \eqref{eq:regularity omega1}.
\end{prop}
\begin{proof}
The Borel $\sigma$-field $\sigma\left\langle\bigcup_{n\in\mathbb{N}}\wiren\right\rangle$ is generated by the $\pi$-system of events
\[\{W\in\wireN:W|_{[n]}=V\},\quad V\in\wiren,\quad n\in\mathbb{N}.\]
For every $n\in\mathbb{N}$ and $V\in\wiren$, we define 
\[\omega(\{W\in\wireN:W|_{[n]}=V\})=\omega^{(n)}(V).\]
By construction, $(\omega^{(n)})_{n\in\Nb}$ is consistent and satisfies
\[\omega^{(m)}(V)=\sum_{V'\in\wiren:V'|_{[m]}=V}\omega^{(n)}(V')\]
for every $m\leq n$ and $V\in\mathcal{W}_{[m]}\setminus\{\mathbf{Id}_m\}$.
Therefore, the function $\omega(\{W\in\wireN:W|_{[m]}=V\})=\omega^{(m)}(V)$ is additive and  Carath\'eodory's extension theorem guarantees a unique extension to a measure on $\wireN\setminus\{\id\}$.  

For each $n\in\mathbb{N}$, $\omega^{(n)}$ determines the jump rates of an exchangeable Markov chain on $\graphsn$, and so $\omega^{(n)}(\wiren\setminus\{\idn\})<\infty$ for all $n\in\mathbb{N}$.  Specifying $\omega(\{\id\})=0$ gives \eqref{eq:regularity omega1}.
Exchangeability of every $\omega_t$, $t\geq0$, makes each $\omega^{(n)}$, $n\in\Nb$, exchangeable and, hence, implies $\omega$ is exchangeable.
\end{proof}

\vspace{2mm}

\noindent{\bf Theorem \ref{thm:Poisson}}. (Poissonian construction).  {\em
Let $\mathbf{\Gamma}=\{\Gbf_G:\,G\in\graphsN\}$ be an exchangeable Feller process on $\graphsN$. 
Then there exists an exchangeable measure $\omega$ satisfying \eqref{eq:regularity omega} such that $\mathbf{\Gamma}\equalinlaw\Gbf_{\omega}^*$, as constructed in \eqref{eq:PPP construction}.
}

\begin{proof}
Let $\omega$ be the exchangeable measure from Proposition \ref{prop:existence omega} and let $\mathbf{W}=\{(t,W_t)\}$ be a Poisson point process with intensity $dt\otimes\omega$.
Since $\omega$ satisfies \eqref{eq:regularity omega1}, Proposition \ref{prop:suff} allows us to construct $\mathbf{\Gamma}^*_{\omega}$ from $\mathbf{W}$ as in \eqref{eq:PPP construction}.  
The jump rates of each $\mathbf{\Gamma}^{*[n]}_{\omega}$ are determined by a thinned version $\mathbf{W}^{[n]}$ of $\mathbf{W}$ that only maintains the atoms $(t,W_t)$ for which $W_t|_{[n]}\neq\idn$.  By the thinning property of Poisson random measures (e.g.\ \cite{KallenbergRM}), the intensity of $\mathbf{W}^{[n]}$ is $\omega^{(n)}$, as defined in \eqref{eq:fidi jump rates}, and it follows immediately that the jump rate from $F$ to $F'\neq F$ in $\mathbf{\Gamma}^{*[n]}_{\omega}$ is
\[\omega^{(n)}(\{W\in\wiren:W(F)=F'\})=Q_n(F,F'),\]
for $Q_n(\cdot,\cdot)$ in \eqref{eq:finite rates}.
Furthermore,
\[Q_n(F,\graphsn\setminus\{F\})=\omega^{(n)}(\{W\in\wiren:W(F)\neq F\})<\infty\]
for all $F\in\graphsn$, $n\in\Nb$.
By construction, $(\mathbf{\Gamma}^{*[n]}_{\omega})_{n\in\Nb}$ is a compatible collection of c\`adl\`ag exchangeable Markov chains governed by the finite-dimensional transition law of $\mathbf{\Gamma}$.
The proof is complete.
\end{proof}

\subsubsection{Standard $\omega$-process}
For $\omega$ satisfying \eqref{eq:regularity omega}, we define the {\em standard $\omega$-process}  $\mathbf{W}^*_{\omega}:=(W_t^*)_{t\geq0}$  on $\wireN$ as the $\wireN$-valued Markov process with $W_0^*=\id$ and infinitesimal jump rates
\[Q(W,dW'):=\omega(\{W''\in\wireN:W''\circ W\in dW'\}),\quad W\neq W'\in\wireN.\]
Associativity of the rewiring maps (Lemma \ref{lemma:existence limit density}(i)) makes $\mathbf{W}=\{\mathbf{W}_V:\,V\in\wireN\}$ a consistent Markov process, where $\mathbf{W}_V:=(W_t^{*}\circ V)_{t\geq0}$ for each $V\in\wireN$.
Thus, an alternative description to the construction of $\Gbf^{*}_{\omega}=\{\Gbf^*_{G,\omega}:\,G\in\graphsN\}$ in \eqref{eq:PPP construction} is to put $\Gbf_{G,\omega}^*=(W_t^*(G))_{t\geq0}$, where $\mathbf{W}^*_{\omega}$ is a standard $\omega$-process.

\begin{cor}\label{cor:Poisson}
Let $\mathbf{\Gamma}=\{\Gbf_G:\,G\in\graphsN\}$ be an exchangeable Feller process on $\graphsN$. 
Then there exists an exchangeable measure $\omega$ satisfying \eqref{eq:regularity omega} such that $\Gbf_G\equalinlaw(W_t^*(G))_{t\geq0}$ for every $G\in\graphsN$, where $\mathbf{W}^*_{\omega}=(W_t^*)_{t\geq0}$ is a standard $\omega$-process.
\end{cor}

\subsection{L\'evy--It\^o representation}

Our main theorem refines Theorem \ref{thm:Poisson} with an explicit L\'evy--It\^o-type characterization of any exchangeable $\omega$ satisfying \eqref{eq:regularity omega1}.
The representation entails a few special measures, which we now introduce.

\subsubsection{Single edge updates}

For $j>i\geq1$ and $k\in\{0,1\}$, we define $\rho^{ij}_{k}:\graphsN\rightarrow\graphsN$ by $G\mapsto G'=\rho^{ij}_{k}(G)$, where
\[
G'^{i'j'}=\left\{\begin{array}{cc}
G^{i'j'},& i'j'\neq ij\\
k,& i'j'=ij.\end{array}\right.\]
In words, $G'=\rho^{ij}_{k}(G)$ coincides with $G$ at every edge except $\{i,j\}$: irrespective of $G^{ij}$, $\rho^{ij}_k$ puts $G'^{ij}=k$.  We call $\rho^{ij}_{k}$ a {\em single edge update map} and define the {\em single edge update measures} by
\begin{equation}\label{eq:reassignment measure}
\epsilon_{k}(\cdot):=\sum_{1\leq i< j<\infty}\delta_{\rho^{ij}_{k}}(\cdot),\quad k=0,1,\end{equation}
which place unit mass at each single edge update map $\rho_k^{ij}$.

\subsubsection{Single vertex updates}\label{section:exile measure}
For any vertex $i\in\mathbb{N}$ and $G\in\graphsN$, the sequence of edges adjacent to $i$  is an infinite $\{0,1\}$-valued sequence $(G^{ij})_{j\neq i}$.
Let $x_0=x_0^1x_0^2\cdots$ and $x_1=x_1^1x_1^2\cdots$ be infinite sequences in $\{0,1\}$ and let $i\in\mathbb{N}$ be any distinguished vertex.
For $x=(x_0,x_1)$ and $i\in\mathbb{N}$, we define $v_x^i:\graphsN\rightarrow\graphsN$  by $G'=v_x^i(G)$, where
\begin{equation}\label{eq:vertex map}
G'^{i'j'}=\left\{\begin{array}{cc} G^{i'j'},& i'\neq i\text{ and }j'\neq i\\
x_0^{j'},& i'=i\text{ and }G^{i'j'}=0\\
x_0^{i'},& j'=i\text{ and }G^{i'j'}=0\\
x_1^{j'},& i'=i\text{ and }G^{i'j'}=1\\
x_1^{i'},& j'=i\text{ and }G^{i'j'}=1
\end{array}\right.,\quad i'\neq j'.
\end{equation}
We call $v_x^i$ a {\em single vertex update map}, as it affects only edges incident to the distinguished vertex $i$.

When $\Gamma$ is the realization of an exchangeable random graph, the sequence $(\Gamma^{ij})_{j\neq i}$ is exchangeable for all fixed $i\in\Nb$.
We ensure that the resulting $\Gamma'=v_x^{i}(\Gamma)$ is exchangeable by choosing $x$ from an exchangeable probability distribution on pairs $X=(X_0,X_1)$ of infinite binary sequences.
Let $\stoch$ denote the space of $2\times 2$ stochastic matrices equipped with the Borel $\sigma$-field, i.e., each $S\in\stoch$ is a matrix $(S_{ii'})_{i,i'=0,1}$ such that
\begin{itemize}
\item $S_{ii'}\geq0$ for all $i,i'=0,1$ and
\item $S_{i0}+S_{i1}=1$ for $i=0,1$.
\end{itemize}
Therefore, each row of $S\in\stoch$ determines a probability measure on $\{0,1\}$.
From a probability measure $\Sigma$ on $\stoch$, we write $W\sim\Omega^{(i)}_{\Sigma}$ to denote the probability measure of a random rewiring map $W\equalinlaw v_{X}^{i}$ constructed by taking $S\sim\Sigma$ and, given $S$, generating $X_0$ and $X_1$ independently of one another according to
\begin{align*}
\mathbb{P}\{X_0^j=k\mid S\}&=S_{0k},\quad k=0,1\quad\text{and}\\
\mathbb{P}\{X_1^j=k\mid S\}&=S_{1k},\quad k=0,1,
\end{align*}
independently for every $j=1,2,\ldots$.
We then define $W=v_X^i$ as in \eqref{eq:vertex map}.
We define the {\em single vertex update measure directed by $\Sigma$} by
\begin{equation}\label{eq:single vertex update}
\Omega_{\Sigma}(\cdot):=\sum_{i=1}^{\infty}\Omega^{(i)}_{\Sigma}(\cdot).\end{equation}

For the reader's convenience, we restate Theorem \ref{thm:Levy-Ito}, whose proof relies on Lemmas \ref{lemma:limit exists}, \ref{lemma:2}, and \ref{lemma:3} from Section \ref{section:key lemmas} below.

\vspace{2mm}

\noindent{\bf Theorem \ref{thm:Levy-Ito}} (L\'evy--It\^o representation). {\em
Let $\Gbf_{\omega}^*=\{\Gbf_{G,\omega}^*:\,G\in\graphsN\}$ be an exchangeable Feller process constructed in \eqref{eq:PPP construction} based on intensity $\omega$ satisfying \eqref{eq:regularity omega}.
Then there exist unique constants $\mathbf{e}_0,\mathbf{e}_1,\mathbf{v}\geq0$, a unique probability measure $\Sigma$ on $2\times2$ stochastic matrices, and a unique measure $\Upsilon$ on $\limitwiredensities$ satisfying
\begin{equation}\label{eq:regularity upsilon}
\Upsilon(\{\mathbf{I}\})=0\quad\text{and}\quad\int_{\limitwiredensities}(1-\upsilon_*^{(2)})\Upsilon(d\upsilon)<\infty
\end{equation}
such that 
\begin{equation}\omega=\Omega_{\Upsilon}+\mathbf{v}\Omega_{\Sigma}+\mathbf{e}_0\epsilon_0+\mathbf{e}_1\epsilon_1.\label{eq:Levy-Ito measure}\end{equation}
}

\subsection{Key lemmas}\label{section:key lemmas}
In the following three lemmas, we always assume that $\omega$ is an exchangeable measure satisfying \eqref{eq:regularity omega1}.
For $n\in\Nb$, we define 
\begin{equation}\label{eq:omega-n}
\omega_n:=\omega\mathbf{1}_{\{W\in\wireN:\,W|_{[n]}\neq\idn\}}\end{equation}
as the restriction of $\omega$ to the event $\{W\in\wireN:\,W|_{[n]}\neq\idn\}$.
While $\omega$ can have infinite mass, the right-hand side of \eqref{eq:regularity omega1} assures that $\omega_n$ is finite for all $n\geq2$.
Exchangeability of $\omega$ implies that $\omega_n$ is invariant with respect to all finite permutations $\mathbb{N}\rightarrow\mathbb{N}$ that fix $[n]$.
Thus, we recover a finite, exchangeable measure from $\omega_n$ by defining the {\em $n$-shift} of $W=(W^{ij})_{i,j\geq1}$ by $\overleftarrow{W}_n:=(W^{n+i,n+j})_{i,j\geq1}$ and letting $\overleftarrow{\omega}_n$ be the image of $\omega_n$ by the $n$-shift, i.e.,
\begin{equation}\label{eq:n-shift measure}
\overleftarrow{\omega}_n(\cdot):=\omega_n(\{W\in\wireN:\,\overleftarrow{W}_n\in\cdot\}).\end{equation}

\begin{lemma}\label{lemma:limit exists}
Suppose $\omega$ is exchangeable and satisfies \eqref{eq:regularity omega1}.
Then $\omega$-almost every $W\in\wireN$ possesses a rewiring limit $|W|$.
\end{lemma}

\begin{proof}
Let $\omega_n$ be the finite measure defined in \eqref{eq:omega-n}.
The image measure $\overleftarrow{\omega}_n$ in \eqref{eq:n-shift measure} is finite, exchangeable, and, therefore, proportional to an exchangeable probability measure on $\wireN$.
By Lemma  \ref{lemma:existence limit density}(iii), $\overleftarrow{\omega}_n$-almost every $W\in\wireN$ possesses a rewiring limit.
Since the rewiring limit of $W\in\wireN$ depends only on $\overleftarrow{W}_n$, for every $n\in\mathbb{N}$, we conclude that $\omega_n$-almost every $W\in\wireN$ possesses a rewiring limit.
Finally, as $n\rightarrow\infty$,
\[\omega_n\uparrow\omega_{\infty}=\omega\mathbf{1}_{\{W\in\wireN:\,W\neq\id\}},\]
the restriction of $\omega$ to $\{W\in\wireN:\,W\neq\id\}$.
By the left-hand side of \eqref{eq:regularity omega1}, $\omega$ assigns zero mass to $\{W\in\wireN:\,W=\id\}$, and so $\omega_{\infty}=\omega$.
The monotone convergence theorem now implies that $\omega$-almost every $W\in\wireN$ possesses a rewiring limit.  
The proof is complete.
\end{proof}

\begin{lemma}\label{lemma:2}
Suppose $\omega$ is exchangeable,  satisfies \eqref{eq:regularity omega1}, and $\omega$-almost every $W\in\wireN$ has $|W|\neq\mathbf{I}$.
Then $\omega=\Omega_{\Upsilon}$ for some unique measure $\Upsilon$ satisfying \eqref{eq:regularity upsilon}.
\end{lemma}

\begin{proof}
Let $\omega_n$ be as in \eqref{eq:omega-n} and for any measurable set $A\subseteq\wireN$ let $\omega_n(\cdot\mid A)$ denote the measure conditional on the event $A$.
For fixed $\upsilon\in\limitwiredensities$ and $A_n:=\{W\in\wireN:\,\overleftarrow{W}_n|_{[2]}\neq\mathbf{Id}_{[2]}\}$, $\omega_n$ satisfies
\begin{eqnarray*}
\omega_n(A_n\mid |W|=\upsilon)&=&\overleftarrow{\omega}_n(\{W|_{[2]}\neq\idtwo\}\mid |W|=\upsilon)\\
&=&\Omega_{\upsilon}(\{W|_{[2]}\neq\idtwo\})\\
&=&1-\upsilon_*^{(2)},
\end{eqnarray*}
where $\upsilon^{(2)}_*$ is the component of $\upsilon\in\limitwiredensities$ corresponding to $\mathbf{Id}_{[2]}$.
By Lemma \ref{lemma:limit exists}, the image of $\omega_n$ by taking rewiring limits, denoted $|\omega_n|$, is well-defined and
\[\omega_n(A_n)=\int_{\wireN}(1-\upsilon_*^{(2)})\omega_n(|W|\in d\upsilon)=\int_{\limitwiredensities}(1-\upsilon^{(2)}_*)|\omega_n|(d\upsilon).\]
By the monotone convergence theorem, $|\omega_n|$ converges to a unique measure $\Upsilon:=|\omega|$ on $\limitwiredensities$.
Since $\omega$-almost every $W\in\wireN$ has $|W|\neq\mathbf{I}$, this limiting measure satisfies $\Upsilon(\{\mathbf{I}\})=0$.
Moreover, 
\[\int_{\limitwiredensities}(1-\upsilon_*^{(2)})|\omega_n|(d\upsilon)\uparrow\int_{\limitwiredensities}(1-\upsilon^{(2)}_*)\Upsilon(d\upsilon)\]
and
\[\omega_n(A_n)\leq\omega(A_n)=\omega(\{W\in\wireN:\,W|_{[2]}\neq\mathbf{Id}_{[2]}\})<\infty.\]
Therefore, $\Upsilon$ satisfies \eqref{eq:regularity upsilon} and $\Omega_{\Upsilon}$ satisfies \eqref{eq:regularity omega1}.

We must still show that $\omega=\Omega_{\Upsilon}$.
By Proposition \ref{prop:mixing measure discrete}, we can write
\[\overleftarrow{\omega}_n(dW)=\int_{\limitwiredensities}\Omega_{\upsilon}(dW)|\overleftarrow{\omega}_n|(d\upsilon),\]
for each $n\in\mathbb{N}$.  
By the monotone convergence theorem and exchangeability, 
\[\lim_{m\uparrow\infty}\overleftarrow{\omega}_m(\{W:\,W|_{[n]}=V\})=\omega(\{W:\,W|_{[n]}=V\}),\]
for every $V\in\wiren\setminus\{\idn\}$, $n\in\Nb$.  Also, for every $m\geq1$,
\[\overleftarrow{\omega}_m(\{W:\,W|_{[n]}=V\})=\int_{\limitwiredensities}\Omega_{\upsilon}(\{W:\,W|_{[n]}=V\})\,|\overleftarrow{\omega}_m|(d\upsilon),\]
which increases to $\Omega_{\Upsilon}(\{W:W|_{[n]}=V\})$ as $m\rightarrow\infty$.
By uniqueness of limits, $\Omega_{\Upsilon}$ and $\omega$ agree on a generating $\pi$-system of the Borel $\sigma$-field and, therefore, they agree on all measurable subsets of $\wireN$.
The proof is complete.
\end{proof}

\begin{lemma}\label{lemma:3}
Suppose $\omega$ is exchangeable, satisfies \eqref{eq:regularity omega1}, and $\omega$-almost every $W\in\wireN$ has $|W|=\mathbf{I}$.
Then there exist unique constants $\mathbf{e}_0,\mathbf{e}_1,\mathbf{v}\geq0$ and a unique probability measure $\Sigma$ on $\stoch$ such that
\[\omega=\mathbf{v}\Omega_{\Sigma}+\mathbf{e}_0\epsilon_0+\mathbf{e}_1\epsilon_1,\]
where $\epsilon_k$ is defined in  \eqref{eq:reassignment measure} and $\Omega_{\Sigma}$ is defined in \eqref{eq:single vertex update}.
\end{lemma}

\begin{proof}
For any $W\in\wireN$, $i\in\Nb$, and $\varepsilon,\delta>0$, we define
\begin{align*}
L_W(i)&:=\limsup_{n\rightarrow\infty}n^{-1}\sum_{j=1}^n\mathbf{1}\{W^{ij}\neq(0,1)\},\\
S_W(\varepsilon)&:=\{i\in\Nb:\,L_W(i)\geq\varepsilon\},\\
|S_W(\varepsilon)|&:=\limsup_{n\rightarrow\infty}n^{-1}\sum_{i=1}^n\mathbf{1}\{i\in S_W(\varepsilon)\},\quad\text{and}\\
V(\varepsilon,\delta)&:=\{W\in\wireN:\,|S_W(\varepsilon)|\geq\delta\}.
\end{align*}
We can partition the event $\{W\in\wireN:\,|W|=\mathbf{I}\}$ by $V\cup E$, where
\begin{align*}
V&:=\bigcup_{i=1}^{\infty}\{W\in\wireN:\,L_W(i)>0\}\quad\text{and}\\
E&:=\bigcap_{i=1}^{\infty}\{W\in\wireN:\,L_W(i)=0\}.
\end{align*}
Thus, we can decompose $\omega$ as the sum of singular measures
\[\omega=\omega_V+\omega_E,\]
where $\omega_V$ and $\omega_E$ are the restrictions of $\omega$ to $V$ and $E$, respectively.
As in \eqref{eq:omega-n}, we write 
\begin{align*}
\omega_{V,n}&:=\omega_V\mathbf{1}_{\{W\in\wireN:\,W|_{[n]}\neq\idn\}}\quad\text{and}\\
\omega_{E,n}&:=\omega_E\mathbf{1}_{\{W\in\wireN:\,W|_{[n]}\neq\idn\}}\
\end{align*}
to denote the restrictions of $\omega$ to $V\cap\{W\in\wireN:\,W|_{[n]}\neq\idn\}$ and $E\cap\{W\in\wireN:\,W|_{[n]}\neq\idn\}$, respectively.
By \eqref{eq:regularity omega1}, both $\omega_{V,n}$ and $\omega_{E,n}$ are finite for all $n\geq2$ and their images under the $n$-shift, denoted $\overleftarrow{\omega}_{V,n}$ and $\overleftarrow{\omega}_{E,n}$, are exchangeable.

\begin{proof}[Case $V$:]
Restricting our attention first to $V$, we note that $|W|=\mathbf{I}$ implies
\[\limsup_{n\rightarrow\infty}n^{-2}\sum_{1\leq i,j\leq n}\mathbf{1}\{W^{ij}\neq(0,1)\}=\lim_{n\rightarrow\infty}n^{-2}\sum_{1\leq i,j\leq n}\mathbf{1}\{W^{ij}\neq(0,1)\}=0.\]
By the law of large numbers for exchangeable sequences, we can replace the upper limits in our definition of $L_W(i)$ and $|S_W(\varepsilon)|$ by proper limits for $\omega$-almost every $W\in\wireN$.
For any $\varepsilon,\delta>0$, $\omega$-almost every $W\in V(\varepsilon,\delta)$ satisfies
\begin{eqnarray*}
\lefteqn{\limsup_{n\rightarrow\infty}n^{-2}\sum_{1\leq i,j\leq n}\mathbf{1}\{W^{ij}\neq(0,1)\}=}\\
&=&\lim_{n\rightarrow\infty}n^{-2}\sum_{1\leq i,j\leq n}\mathbf{1}\{W^{ij}\neq(0,1)\}\\
&=&\lim_{n\rightarrow\infty}\lim_{m\rightarrow\infty}(nm)^{-1}\sum_{i=1}^n\sum_{j=1}^m\mathbf{1}\{W^{ij}\neq(0,1)\}\\
&=&\lim_{n\rightarrow\infty}n^{-1}\sum_{i=1}^n\lim_{m\rightarrow\infty}m^{-1}\sum_{j=1}^m\mathbf{1}\{W^{ij}\neq(0,1)\}\\
%&\geq&\lim_{n\rightarrow\infty}\lim_{m\rightarrow\infty}(nm)^{-1}\sum_{i=1}^n\mathbf{1}\{i\in S_W(\varepsilon)\} \sum_{j=1}^m\mathbf{1}\{W^{ij}\neq(0,1)\}\\
%&\geq&\lim_{n\rightarrow\infty}n^{-1}\sum_{i=1}^{n}\mathbf{1}\{i\in S_W(\varepsilon)\}\lim_{m\rightarrow\infty}m^{-1}\sum_{j=1}^m\mathbf{1}\{i\in S_W(\varepsilon),\,W^{ij}\neq(0,1)\}\\
&\geq&\lim_{n\rightarrow\infty}n^{-1}\sum_{i=1}^n \varepsilon\mathbf{1}\{i\in S_W(\varepsilon)\}\\
%&\geq&\lim_{n\rightarrow\infty}n^{-1}\sum_{i=1}^n\mathbf{1}\{i\in S_W(\varepsilon)\}\varepsilon\\
%&\geq&\lim_{n\rightarrow\infty}n^{-2}\sum_{1\leq i,j\leq n}\mathbf{1}\{i\in S_W(\varepsilon),\,W^{ij}\neq(0,1)\}\\
%&=&\lim_{n\rightarrow\infty}\left(n^{-1}\sum_{i=1}^n\mathbf{1}\{i\in S_W(\varepsilon)\}\right)\left(n^{-1}\sum_{j=1}^{n}\mathbf{1}\{i\in S_W(\varepsilon),\,W^{ij}\neq(0,1)\}\right)\\
&\geq&\varepsilon\delta\\
&>&0,
\end{eqnarray*}
contradicting the assumption that $\omega$-almost every $W\in\wireN$ has $|W|=\mathbf{I}$.
(In the above string of inequalities, passage from the second to third line follows from exchangeability and the law of large numbers and the interchange of sum and limit in the fourth line is permitted because the sum is finite and each of the limits exists by exchangeability and the law of large numbers.)
Therefore, $\omega$ assigns zero mass to $V(\varepsilon,\delta)$ for all $\varepsilon,\delta>0$, and so $\omega$-almost every $W\in\wireN$ must have $|S_W(\varepsilon)|=0$ for all $\varepsilon>0$.
In this case, either $\#S_W(\varepsilon)<\infty$ or $\#S_W(\varepsilon)=\infty$.

Treating the latter case first, we define the $n$-shift $\overleftarrow{W}_n$ of $W\in\wireN$ as above, so that $\overleftarrow{\omega}_{V,n}$ is finite, exchangeable, and assigns positive mass to the event
\[\{W\in\wireN:\,|S_W(\varepsilon)|=0\text{ and }\#S_{\overleftarrow{W}_n}(\varepsilon)=\infty\}.\]
We can regard $\overleftarrow{\omega}_{V,n}$ as a constant multiple of an exchangeable probability measure $\theta_n$ on $\wireN$, so that for fixed $n\in\Nb$ the sequence $(\mathbf{1}\{L_{\overleftarrow{W}_n}(i)\geq\varepsilon\})_{i\in\Nb}$
is an exchangeable $\{0,1\}$-valued sequence under $\theta_n$.
By de Finetti's theorem, 
\[|S_W(\varepsilon)|=\lim_{m\rightarrow\infty}m^{-1}\sum_{i=1}^m\mathbf{1}\{L_{\overleftarrow{W}_n}(i)\geq\varepsilon\}=0\quad\theta_n\text{-a.s.,}\]
which implies $L_{\overleftarrow{W}_n}(i)=0$ for all $i\in\Nb$ $\overleftarrow{\omega}_{V,n}$-a.e., a contradiction.

Now consider the case $\#S_{W}(\varepsilon)<\infty$.
We claim that $\#S_W(\varepsilon)=1$ almost surely.
Suppose $\#S_W(\varepsilon)=k\geq2$ so that $\{i_1<\cdots<i_k\}$ is the list of indices for which $L_W(i_j)\geq\varepsilon$.
Taking $n=i_k-1$, we know that ${\omega}_{V,n}$ is a finite measure that is invariant under permutations that fix $[n]$.
Under the $n$-shift, $\overleftarrow{\omega}_{V,n}$ is finite, exchangeable, and assigns equal mass to events of the form
\[A_i=\{W\in\wireN:\,L_{\overleftarrow{W}_n}(i)\geq\varepsilon,\,L_{\overleftarrow{W}_n}(j)<\varepsilon\text{ for }j\neq i\},\quad i\in\Nb.\]
Thus, if $\overleftarrow{\omega}_{V,n}(A_i)>0$ for any $i\in\Nb$, then $\overleftarrow{\omega}_{V,n}$ has infinite total mass.
Since this argument applies as long as $\#S_W(\varepsilon)=k\geq2$, it follows that $\omega$ must assign zero mass to $\{W\in\wireN:\,\#S_W(\varepsilon)>1\}$.

On the other hand, if $\#S_W(\varepsilon)=1$, then $\omega$ must assign the same mass to all events
\[A_i=\{W\in\wireN:\,L_W(i)\geq\varepsilon,\,L_W(j)<\varepsilon\text{ for all }j\neq i\}.\]
In this case,
\[\omega_V(\{W\in\wireN:\,W|_{[n]}\neq\idn\})\leq\sum_{i=1}^{n}\omega_{V,n}(A_i)=n\omega_{V,n}(A_1)<\infty,\quad\text{for all }n\in\Nb,\]
which does not contradict \eqref{eq:regularity omega}.

Since the above holds for all $\varepsilon>0$, there must be a unique vertex $i\in\Nb$ for which $L_W(i)>0$.
By exchangeability and the right-hand side of \eqref{eq:regularity omega1}, $(W^{ij})_{j\neq i}$ is governed by a finite measure $\mathbf{v}_i\Sigma_i$, where $\mathbf{v}_i\geq0$ and $\Sigma_i$ is the unique probability measure guaranteed by de Finetti's theorem.
Thus,
\[\omega_V=\sum_{i=1}^{\infty}\mathbf{v}_i\Omega_{\Sigma_i}^{(i)}.\]
Exchangeability of $\omega$ requires that each term is equal, so there is a unique probability measure $\Sigma$ and unique $\mathbf{v}\geq0$ for which
\[\omega_{V}=\mathbf{v}\sum_{i=1}^{\infty}\Omega_{\Sigma}^{(i)}.\]

Just as any exchangeable rewiring map determines a transition kernel from $\graphsN$ to $\graphsN$, any exchangeable $\{0,1\}\times\{0,1\}$-valued sequence determines an exchangeable transition kernel from $\{0,1\}$ to $\{0,1\}$: let $W=(W_0^i,W_1^i)_{i\geq1}$ be $\{0,1\}\times\{0,1\}$-valued and for $x=(x^i)_{i\geq1}$ in $\{0,1\}^{\Nb}$ define $x'=W(x)$ by
\[x'^i=\left\{\begin{array}{cc}
W_0^i,& x^i=0\\
W_1^i,& x^i=1.
\end{array}\right.\]
A straightforward argument along the same lines as Theorem \ref{thm:discrete-rep} (with the obvious substitution of de Finetti's theorem for Aldous--Hoover) puts exchangeable $\{0,1\}\times\{0,1\}$-valued sequences in correspondence with exchangeable transition kernels on $\{0,1\}^{\Nb}$.
The ergodic measures in this space are in correspondence with the unit square $[0,1]\times[0,1]$ and, thus, also the space of $2\times 2$ stochastic matrices, allowing us to regard $\Sigma$ as a probability measure on $\mathcal{S}_2$.
%{\color{blue}
%Since the space of exchangeable, dissociated probability measures on $\{0,1\}\times\{0,1\}$-valued sequences is isomorphic to the space of $2\times 2$ stochastic matrices, we can regard $\Sigma$ as a probability measure on $\stoch$.
%}
\end{proof}

\begin{proof}[Case $E$:]
On event $E$, $W^{ij}\neq(0,1)$ occurs for a zero proportion of all pairs $ij$ and also a zero proportion of all $j\neq i$ for a fixed $i$.
For $W\in\wireN$, we define 
\[E_W:=\{ij:\,W^{ij}\neq(0,1)\},\]
which must satisfy either $\#E_W=\infty$ or $\#E_W<\infty$.
Also, since $\overleftarrow{\omega}_{E,n}$ is finite, we can write $\overleftarrow{\omega}_{E,n}\propto\theta_n$ for some exchangeable probability measure $\theta_n$ on $\wireN$, for each $n\geq2$.

Suppose first that $\#E_W=\infty$.
Then $\overleftarrow{\omega}_{E,2}<\infty$ implies $\overleftarrow{W}_2\sim\theta_2$ is an exchangeable $\{0,1\}$-valued array with 
\[\limsup_{n\rightarrow\infty}n^{-2}\sum_{1\leq i,j\leq n}\mathbf{1}\{\overleftarrow{W}_2^{ij}\neq(0,1)\}=0.\]
By the Aldous--Hoover theorem, $\#E_{\overleftarrow{W}_2}=0$ $\overleftarrow{\omega}_{E,2}$-almost everywhere, forcing all edges in $E_W$ to be in the first row of $W$.
As $m\rightarrow\infty$, $\overleftarrow{\omega}_{E,m}$ forces all edges of $E_W$ to be in the first $m-1$ rows of $W\in\wireN$.
By exchangeability, all elements of $E_W$ must be in the same row for $\omega_E$-almost every $W\in\wireN$.
Thus, $\omega_{E}$ assigns equal mass to each event $R_i\subseteq E$, where
\[R_i:=\{W\in\wireN:\,i'j'\in E_W\quad\text{if and only if}\quad i\in\{i',j'\}\}\]
is the event that all non-trivial entries of $W$ are in the $i$th row.

For every $n\geq2$, $\omega_{E,n}$ is exchangeable with respect to permutations that fix $[n]\rightarrow[n]$, and so $(W^{n,n+j})_{j\geq1}$ is exchangeable for each $n\geq1$ and 
\[\limsup_{m\rightarrow\infty}m^{-1}\sum_{j=1}^m\mathbf{1}\{W^{n,n+j}\neq(0,1)\}=0.\]
By de Finetti's theorem, $W^{n,n+j}=0$ for all $j\geq1$ almost surely, and so $\omega_{E,n}$-almost every $W\in\wireN$ must have  $\#E_W<\infty$, for every $n\geq2$.
By the monotone convergence theorem, $\omega_E$-almost every $W$ must have $\#E_W<\infty$.

Now suppose that $\#E_W<\infty$.
Then $\omega_{E}=\sum_{i=1}^{\infty}\omega_{E,R_i}$, where $\omega_{E,R_i}$ is the restriction of $\omega_{E}$ to $R_i$ for every $i\in\Nb$.
By exchangeability, every $\omega_{E,R_i}$ must determine the same exchangeable measure on $\{0,1\}$-valued sequences.
By a similar argument as for Case $V$ above, we can rule out all possibilities except the event that $W^{ij}\neq(0,1)$ for a unique pair $ij$, $i\neq j$.
For suppose $\#E_W=k\geq2$, then there is a unique $i\in\Nb$ such that
\[i'j'\in E_W\quad\text{if and only if}\quad i\in\{i',j'\}.\]
Without loss of generality, we assume $i=1$ so that we can identify $E_W=\{j\in\Nb:\,W^{1j}\neq(0,1)\}=\{1\leq i_1<\cdots<i_k\}$ with a finite subset of $\Nb$.
We define $n=i_k-1$ and $X_{W,n}=(\mathbf{1}\{W^{1,n+j}\neq(0,1)\})_{j\geq1}$.
By assumption, $\overleftarrow{\omega}_{E,n}$ is finite and exchangeable, and so $X_{W,n}$ is an exchangeable $\{0,1\}$-valued sequence with only a single non-zero entry.
By exchangeability, $\overleftarrow{\omega}_{E,n}$ assigns the same mass to all sequences in $\{X_{W,n}^{(n+1,n+j)}:\,j\geq1\}$, where $(n+1,n+j)$ represents the transposition of elements $n+1$ and $n+j$.
It follows that $X_{W,n}$ can have positive mass only if $\overleftarrow{\omega}_{E,n}$ has infinite total mass, a contradiction.

On the other hand, if $\#E_W=1$, then we can express $\omega_E$ as
\[\omega_E=\sum_{j>i\geq1}(\mathbf{c}_{0}^{ij}\delta_{\rho^{ij}_{0}}+\mathbf{c}_{1}^{ij}\delta_{\rho^{ij}_{0}}+\mathbf{c}_{10}^{ij}\delta_{\rho^{ij}_{10}}),\]
where $\mathbf{c}_0^{ij},\mathbf{c}_1^{ij},\mathbf{c}_{10}^{ij}\geq0$, $\rho^{ij}_{0}$ and $\rho^{ij}_1$ are the single edge update maps defined before \eqref{eq:reassignment measure}, and $\rho^{ij}_{10}$ is the rewiring map $W$ with all $W^{i'j'}=(0,1)$ except $W^{ij}=(1,0)$.
This measure clearly assigns zero mass to $\id$, but it is exchangeable only if there are $\mathbf{c}_0,\mathbf{c}_1,\mathbf{c}_{10}\geq0$ such that $\mathbf{c}_0^{ij}=\mathbf{c}_0$, $\mathbf{c}_1^{ij}=\mathbf{c}_1$, and $\mathbf{c}_{10}^{ij}=\mathbf{c}_{10}$ for all $j>i\geq1$.
This measure is exchangeable and satisfies the right-hand side of \eqref{eq:regularity omega}.
Since $\rho_{10}^{ij}$ gives $\rho_{10}^{ij}(W)=\rho_{1}^{ij}(W)$ whenever $W^{ij}=0$ and $\rho_{10}^{ij}(W)=\rho_0^{ij}(W)$ when $W^{ij}=1$, we define $\mathbf{e}_0=\mathbf{c}_{0}+\mathbf{c}_{10}$ and $\mathbf{e}_1=\mathbf{c}_{1}+\mathbf{c}_{10}$.
The proof is complete.

\end{proof}

\end{proof}

\begin{proof}[Proof of Theorem \ref{thm:Levy-Ito}]
%By Theorem \ref{thm:Poisson}, the evolution of $\mathbf{\Gamma}$ is determined by a measure $\omega$ satisfying \eqref{eq:regularity omega1}.
By Lemma \ref{lemma:limit exists}, $\omega$-almost every $W\in\wireN$ possesses a rewiring limit $|W|$ and we can express $\omega$ as a sum of singular components:
\[\omega=\omega\mathbf{1}_{\{|W|\neq\mathbf{I}\}}+\omega\mathbf{1}_{\{|W|=\mathbf{I}\}}.\]
By Lemma \ref{lemma:2}, the first term corresponds to $\Omega_{\Upsilon}$ for some unique measure $\Upsilon$ satisfying \eqref{eq:regularity upsilon}.
By Lemma \ref{lemma:3}, the second term corresponds to $\mathbf{v}\Omega_{\Sigma}+\mathbf{e}_0\epsilon_0+\mathbf{e}_1\epsilon_1$ for a unique probability measure $\Sigma$ on $\stoch$ and unique constants $\mathbf{v},\mathbf{e}_0,\mathbf{e}_1\geq0$.
The proof is complete.
\end{proof}

\subsection{The projection into $\limitdensities$}
Let $\Gbf=\{\Gbf_G:\,G\in\graphsN\}$ by an exchangeable Feller process and, for $D\in\limitdensities$, recall the definition of $\Gbf_D=(\Gamma_t)_{t\geq0}$ as the process obtained by first taking $\Gamma_0\sim\gamma_D$ and then putting $\Gbf_D=\Gbf_G$ on the event $\Gamma_0=G$.
We now show that the projection $|\mathbf{\Gamma}_{D}|=(|\Gamma_t|)_{t\geq0}$ into $\limitdensities$ exists simultaneously for all $t\geq0$ with probability one.
We also prove that these projections determine a Feller process on $\limitdensities$ by
\[|\Gbf_{\limitdensities}|=\{|\Gbf_{D}|:\,D\in\limitdensities\}.\]

\begin{prop}\label{prop:density fixed t}
Let $\Gbf=\{\Gbf_G:\,G\in\graphsN\}$ be an exchangeable Feller process on $\graphsN$.
For any $D\in\limitdensities$, let $\Gbf_D=(\Gamma_t)_{t\geq0}$ be obtained by taking $\Gamma_0\sim\gamma_D$ and putting $\Gbf_{D}=\Gbf_G$ on the event $\Gamma_0=G$.
Then the graph limit $|\Gamma_t|$ exists almost surely for every $t\geq0$.
\end{prop}
\begin{proof}
For each $D\in\limitdensities$, this follows immediately by exchangeability of $\Gamma_0\sim\gamma_D$, exchangeability of the transition law of $\Gbf$, and the Aldous--Hoover theorem.
\end{proof}

Proposition \ref{prop:density fixed t} only establishes that the graph limits exist at any countable collection of times.   
Theorem \ref{thm:induced density process} establishes that $|\Gamma_t|$ exists simultaneously for all $t\geq0$.

\vspace{2mm}

\noindent{\bf Theorem \ref{thm:induced density process}}.  {\em 
Let $\mathbf{\Gamma}=\{\Gbf_G:\,G\in\graphsN\}$ be an exchangeable Feller process on $\graphsN$.
Then $|\mathbf{\Gamma}_{\limitdensities}|$ exists almost surely and is a Feller process on $\limitdensities$.
Moreover, each $|\mathbf{\Gamma}_D|$ is continuous at all $t>0$ except possibly at the times of Type (III) discontinuities.
}

\vspace{2mm}

Recall that every $D\in\limitdensities$ corresponds to a probability measure $\gamma_D$ on $\graphsN$.  We equip $\limitdensities$ with the metric \eqref{eq:TV},
under which it is a  compact space.  
Furthermore, any $\upsilon\in\limitwiredensities$ corresponds to a transition probability $P_{\upsilon}$ on $\graphsN\times\graphsN$ and, thus, determines a
Lipschitz continuous map $\upsilon:\limitdensities\rightarrow\limitdensities$ through $D\mapsto D\upsilon$ as in \eqref{eq:right action}.
Consequently,
\[\xnorm{D{\upsilon}-D'\upsilon}\leq \xnorm{D-D'}\quad\text{for all }D,D'\in\limitdensities\text{ and all }\upsilon\in\limitwiredensities.\]

\begin{proof}[Proof of Theorem \ref{thm:induced density process}]

\begin{proof}[Existence]
Fix any $D\in\limitdensities$ and let $\Gbf_{[0,1]}=(\Gamma_{[0,1]}^{ij})_{i,j\geq1}$ denote the array of edge trajectories in $\Gbf_D$ on the interval $[0,1]$.
By the consistency assumption for $\Gbf$, each $\Gamma_{[0,1]}^{ij}$ is an alternating collection of finitely many $0$s and $1$s.
 Formally, each $\Gamma_{[0,1]}^{ij}=y$ is a partition of $[0,1]$ into finitely many non-overlapping intervals $J_l$ along with an initial status $y_0\in\{0,1\}$.  
The starting condition $y_0$ determines the entire path $y=(y_t)_{t\in[0,1]}$, because $y$ must alternate between states $0$ and $1$ in the successive subintervals $J_l$.  

We denote this space by $\mathcal{I}$, and we write $\mathcal{I}^*$ to denote the closure of $\mathcal{I}$ in the Skorokhod space of c\`adl\`ag functions $[0,1]\rightarrow\{0,1\}$.
We can partition $\mathcal{I}^*:=\bigcup_{m\in\mathbb{N}}\mathcal{I}^*_m$, where $\mathcal{I}^*_m$ is the closure of
\[\{y\in\mathcal{I}^*:\,y\text{ has exactly } m\text{ sub-intervals}\}.\]
Consequently, $\mathcal{I}^*$ is complete, separable, and Polish.
The fact that $\mathcal{I}^*$ is Polish allows us to apply the Aldous--Hoover theorem to study the $\mathcal{I}^*$-valued array $\Gbf_{[0,1]}$.

Since the initial distribution of $\Gbf_D$ is exchangeable, the entire path $(\Gamma_t)_{t\geq0}$ satisfies
\[(\Gamma_t^\sigma)_{t\geq0}\equalinlaw(\Gamma_t)_{t\geq0}\quad\text{ for all permutations }\sigma:\Nb\rightarrow\Nb.\]
In particular, the array $\Gbf_{[0,1]}$ of edge trajectories  is weakly exchangeable.
By the Aldous--Hoover theorem, we can express $\Gbf_{[0,1]}$ as a mixture of exchangeable, dissociated $\mathcal{I}^*$-valued arrays.
Thus, it suffices to prove existence of $|Y|=(|Y_t|)_{t\in[0,1]}$ for exchangeable, dissociated $\mathcal{I}^*$-valued arrays $Y=(Y^{ij})_{i,j\geq1}$.

To this end, we treat $Y$ just as the edge trajectories of a graph-valued process on $[0,1]$, so $Y^{ij}_t$ denotes the status of edge $ij$ at time $t\in[0,1]$ and $\delta(F,Y_t)$ is the limiting density of any finite subgraph $F\in\mathcal{G}^*$ in $Y_t$.
To show that $(|Y_t|)_{t\in[0,1]}$ exists simultaneously at all $t\in[0,1]$, we show that the upper and lower limits of
$(\delta(F,Y^{[n]}_t))_{n\in\Nb}$ coincide for all $t\in[0,1]$ simultaneously.  
In particular, for fixed $F\in\mathcal{G}_m$, $m\in\mathbb{N}$, and $t\in[0,1]$ we write
\begin{align*}
\delta_t^+(F)&:=\limsup_{n\rightarrow\infty}\delta(F,Y^{[n]}_t)\quad\text{and}\\
\delta_t^-(F)&:=\liminf_{n\rightarrow\infty}\delta(F,Y^{[n]}_t).\end{align*}
We already know, cf.\ Proposition \ref{prop:density fixed t}, that $\mathbb{P}\{\delta_t^+(F)=\delta_t^-(F)\}=1$ for all {\em fixed} $t\in[0,1]$, but we wish to show that $\delta_t^+(F)=\delta_t^-(F)$ almost surely for {\em all} $t\in[0,1]$, i.e., 
\[\mathbb{P}\{\sup_{t\in[0,1]} |\delta_t^+(F)-\delta_t^-(F)|=0\}=1.\]  

Each path $Y^{ij}_{[0,1]}=(Y^{ij}_t)_{t\in[0,1]}$ has finitely many discontinuities almost surely and so must $Y^{[n]}_{[0,1]}:=(Y^{ij}_{[0,1]})_{1\leq i,j\leq n}$, for every $n\in\mathbb{N}$.
For every $\varepsilon>0$, there is a finite subset $S_{\varepsilon}\subset[0,1]$ and an at most countable partition $J_1,J_2,\ldots$ of the open set $[0,1]\setminus S_{\varepsilon}$ such that 
\[
\begin{array}{l}
\mathbb{P}\{Y^{ij}_{[0,1]}\text{ is discontinuous at }s\}\geq\varepsilon\quad\text{for every }s\in S_{\varepsilon}\text{ and}\\
\mathbb{P}\{Y^{ij}_{[0,1]}\text{ is discontinuous in }J_l\}<\varepsilon,\quad l=1,2,\ldots.
\end{array}\]
%Such a partition must exist by the strong law of large numbers.
Without such a partition, there must be a sequence of intervals $(t-\rho_n,t+\rho_n)$ with $\rho_n\rightarrow0$ that converges to $t\notin S_{\varepsilon}$ such that 
\[\mathbb{P}\{Y_{[0,1]}^{ij}\text{ is discontinuous in }(t-\rho_n,t+\rho_n)\}\geq\varepsilon\quad\text{ for every }n\geq1.\]
Continuity from above implies  
\[\mathbb{P}\{Y_{[0,1]}^{ij}\text{ is discontinuous at }t\notin S_{\varepsilon}\}\geq\varepsilon>0,\]
 which contradicts the assumption $t\notin S_{\varepsilon}$.

The strong law of large numbers also implies that 
\[\mathbb{P}\{Y^{12}\text{ is discontinuous in }J_l\}=\lim_{n\rightarrow\infty}\frac{2}{n(n-1)}\sum_{1\leq i<j\leq n}\mathbf{1}\{Y^{ij}\text{ is discontinuous in }J_l\}<\varepsilon\quad\text{a.s.}\]
 for each sub-interval $J_l$, $l=1,2,\ldots$.
By ergodicity of the action of $\symmetricN$ for dissociated arrays, $\delta_t^+(F)$ and $\delta_t^-(F)$ cannot vary by more than $\varepsilon$ over $J_l$ and, since $\delta_t^+(F)=\delta_t^-(F)$ almost surely for each endpoint of $J_l$,  
\[\mathbb{P}\{\sup_{t\in J_l}|\delta_{t}^+(F)-\delta_t^-(F)|\leq2\varepsilon\}=1\]
for all $l=1,2,\ldots$ and all $\varepsilon>0$.

Since $[0,1]$ is covered by at most countably many sub-intervals $J_l$, $l=1,2,\ldots$, and the non-random set $S:=\bigcup_{\varepsilon>0}S_{\varepsilon}$, it follows that
\[\mathbb{P}\{\sup_{t\in[0,1]}|\delta_t^+(F)-\delta_t^-(F)|\leq2\varepsilon\}=1\quad\text{for all }\varepsilon>0,\text{ for every }F\in\mathcal{G}^*.\]
Continuity from above implies that 
\[\mathbb{P}\{\sup_{t\in[0,1]}|\delta_t^+(F)-\delta_t^-(F)|=0\}=\lim_{\varepsilon\downarrow0}\mathbb{P}\{\sup_{t\in[0,1]}|\delta_t^+(F)-\delta_t^-(F)|\leq2\varepsilon\}=1;\]
thus, $\mathbb{P}\{\delta(F,Y_t)\text{ exists at all }t\in[0,1]\}=1$ for every $F\in\bigcup_{m\in\mathbb{N}}\mathcal{G}_{[m]}$.  
Countable additivity of probability measures implies that this limit exists almost surely for all $F\in\bigcup_{m\in\mathbb{N}}\mathcal{G}_m$ and, thus, $Y_{[0,1]}$ determines a process on $\limitdensities$ almost surely.  

The fact that these limits are deterministic follows from the 0-1 law and our assumption that $Y$ is dissociated.
%: the homomorphism densities of any $\Gamma_t$ depend only on the tail $\sigma$-field generated by $(\Gamma_{t}|_{\mathbb{N}\setminus[n]})_{n\geq1}$.
%For fixed $F\in\mathcal{G}^*$, the sequence $(\delta(F,\Gamma_t|_{[n]}))_{n\geq1}$ satisfies
%\[E(\delta(F,\Gamma_t|_{[m]})\mid (\delta(F,\Gamma_t|_{[n']}))_{n'\geq n})=\delta(F,\Gamma_t|_{[n]})\quad \text{for all }n\geq m\]
%and is, therefore, a reverse martingale.
%The reverse martingale convergence theorem implies that $\delta(F,\Gamma_t)=\lim_{n\rightarrow\infty}\delta(F,\Gamma_t|_{[n]})$ exists almost surely.
%Almost sure existence of the limit follows by a straightforward martingale argument, as the sequence $(\delta(F,\Gamma_t|_{[n]}))_{n\geq1}$ is a reverse martingale for every fixed $F\in\mathcal{G}^*$.
By the Aldous--Hoover theorem, $\Gbf_D$ is a mixture of exchangeable, dissociated processes, so we conclude that $|\Gbf_{D}|$ exists almost surely for every $D\in\limitdensities$.
\end{proof}

\begin{proof}[Feller property]

By Corollary \ref{cor:Poisson}, we can assume $\Gbf=\{\Gbf_G:\,G\in\graphsN\}$ is constructed by $\Gbf_G=(W_t^*(G))_{t\geq0}$, for each $G\in\graphsN$, where $\mathbf{W}_{\omega}=(W_t^*)_{t\geq0}$ is a standard $\omega$-process.
The standard $\omega$-process has the Feller property and, by Theorem \ref{thm:induced density chain}, 
\[|\Gamma_t|\equalinlaw|W_t^*(\Gamma_0)|=|\Gamma_0|{|W_t^*|} \quad\text{a.s.\ for every }t\geq0.\]  
Existence of $|W_t^*|$ at fixed times follows by exchangeability and analog to Proposition \ref{prop:density fixed t}.
In fact, $|\mathbf{W}_{\omega}^*|=(|W_t|)_{t\geq0}$ exists for all $t\geq0$ simultaneously by analog to the above argument for existence of $|\Gbf_D|$, because $\id$ is an exchangeable initial state with rewiring limit $\mathbf{I}$.
The Markov property of $|\Gbf_{\limitdensities}|$ follows immediately.

Let $(\mathbf{P}_t)_{t\geq0}$ be the semigroup of $|\mathbf{\Gamma}_{\limitdensities}|$, i.e., for every continuous function $g:\limitdensities\rightarrow\mathbb{R}$
\[\mathbf{P}_tg(D)=\mathbb{E}(g(|\Gamma_t|)\mid|\Gamma_0|=D).\]
To establish the Feller property, we must show that for every continuous $g:\limitdensities\rightarrow\mathbb{R}$
\begin{itemize}
	\item[(i)] $D\mapsto\mathbf{P}_tg(D)$ is continuous for all $t>0$.
	\item[(ii)] $\lim_{t\downarrow0}\mathbf{P}_tg(D)=g(D)$ for all $D\in\limitdensities$ and
\end{itemize}

For (i), we take any continuous function $h:\limitdensities\rightarrow\mathbb{R}$.  By compactness of $\limitdensities$, $h$ is uniformly continuous and, hence, bounded.  
By dominated convergence and Lipschitz continuity of $|W_t|:\limitdensities\rightarrow\limitdensities$, $D\mapsto\mathbf{P}_th(D)$ is continuous for all $t>0$.  

Part (ii) follows from exchangeability of $\mathbf{W}^*_{\omega}$ and the Feller property of $\mathbf{\Gamma}$.
To see this explicitly, we show the equivalent condition 
\[\lim_{t\downarrow0}\mathbb{P}\{\xnorm{|W_t^*|-\mathbf{I}}>\varepsilon\}=0\quad\text{for all }\varepsilon>0.\]
For contradiction, we assume
\[\limsup_{t\downarrow0}\mathbb{P}\{\xnorm{|W_t^*|-\mathbf{I}}>\varepsilon\}>0\quad\text{for some }\varepsilon>0.\]

%Fix an enumeration $V_1,V_2,\ldots$ of $\mathcal{W}^*=\bigcup_{m\in\Nb}\mathcal{W}_{[m]}$.
By definition of the metric \eqref{eq:TV} on $\limitdensities$,
\[\xnorm{|W_t^*|-\mathbf{I}}=\sum_{n\in\Nb}2^{-n}\sum_{V\in\wiren}|\,|W_t^*|(V)-\mathbf{I}(V)|,\]
and so there must be some $n\geq2$, some $V\in\wiren\setminus\{\idn\}$, and $\varepsilon,\varrho>0$ such that 
\begin{equation}\label{eq:contra}
\limsup_{t\downarrow0}\mathbb{P}\{\delta(V,W_t)>\varepsilon\}\geq\varrho.\end{equation}
Given such a $V\in\wiren\setminus\{\idn\}$, we can choose $F\in\graphsn$ such that $V(F)\neq F$ and define
\[\psi_F(\cdot):=\mathbf{1}\{\cdot|_{[n]}\neq F\}\]
so that $\mathbf{P}_t\psi_F(G)=\mathbb{P}\{\Gamma_t^{[n]}\neq F\mid\Gamma_0=G\}$.
We choose any $G=F^*\in\graphsN$ such that $F^*|_{[n]}=F$.
In this way,
\[\mathbf{P}_t\psi_F(F^*)=\mathbb{P}\{\Gamma_t^{[n]}\neq F\mid\Gamma_0=F^*\}\leq\mathbb{P}\{\mathbf{\Gamma}^{[n]}_{F^*}\text{ discontinuous on }[0,t]\}\]
and
\[\mathbb{P}\{\Gamma_t^{[n]}\neq F\mid\Gamma_0=F^*\}\geq\mathbb{P}\{W_t^{*[n]}=V\}=\mathbb{E}\mathbb{P}\{W_t^{*[n]}=V\mid|W_t^*|\}=\mathbb{E}\delta(V,W_t^*).\]
Now,
\begin{align*}
\mathbb{P}\{\mathbf{\Gamma}^{[n]}_{F^*}\text{ discontinuous on }[0,t]\}\leq1-\exp\{-t\omega(W:\,W|_{[n]}\neq\idn)\}
\end{align*}
implies
\[\limsup_{t\downarrow0}\mathbb{P}\{\mathbf{\Gamma}^{[n]}_{F^*}\text{ discontinuous on }[0,t]\}\leq \limsup_{t\downarrow0}1-\exp\{-t\omega(W:\,W|_{[n]}\neq\idn)\}=0,\]
because $\omega(W:\,W|_{[n]}\neq\idn)<\infty$ for all $n\geq2$ by the right-hand side of \eqref{eq:regularity omega1}.
On the other hand, Markov's inequality implies
\[\mathbb{E}\delta(V,W_t^*)\geq\varepsilon\mathbb{P}\{\delta(V,W_t^*)>\varepsilon\},\]
and \eqref{eq:contra} gives
\[\limsup_{t\downarrow0}\mathbb{E}\delta(V,W_t^*)\geq\varepsilon\limsup_{t\downarrow0}\mathbb{P}\{\delta(V,W_t^*)>\varepsilon\}\geq\varepsilon\varrho>0.\]
Thus, the Feller property of $\Gbf$ and the hypothesis \eqref{eq:contra} lead to contradicting statements
\begin{align*}
\limsup_{t\downarrow0}\mathbf{P}_t\psi_F(F^*)&\geq\varepsilon\varrho>0\quad\text{and}\\
\limsup_{t\downarrow0}\mathbf{P}_t\psi_F(F^*)&\leq0.
\end{align*}
We conclude that
\[\limsup_{t\downarrow0}\mathbb{P}\{\delta(V,W_t^*)>\varepsilon\}<\varrho\]
for all $V\in\wiren\setminus\{\idn\}$, $n\in\Nb$, and all $\varrho,\varepsilon>0$.
There are $4^{n-1\choose2}-1<4^{n\choose2}$ elements in $\wiren\setminus\{\idn\}$ for each $n\geq2$.
Thus,
\[\mathbb{P}\{\delta(\idn,W_t^*)<1-\varepsilon\}\leq\mathbb{P}\left\{\bigcup_{V\in\wiren\setminus\{\idn\}}\{\delta(V,W_t^*)>\varepsilon4^{-{n\choose2}}\}\right\},\]
and we have
\begin{eqnarray*}
\lefteqn{\limsup_{t\downarrow0}\mathbb{P}\{\delta(\idn,W_t^*)<1-\varepsilon\}\leq}\\
&\leq&\limsup_{t\downarrow0}\sum_{V\in\wiren\setminus\{\idn\}}\mathbb{P}\{\delta(V,W_t^*)>\varepsilon4^{-{n\choose2}}\}\\
&\leq&\sum_{V\in\wiren\setminus\{\idn\}}\limsup_{t\downarrow0}\mathbb{P}\{\delta(V,W_t^*)>\varepsilon4^{-{n\choose2}}\}\\
&=&0.
\end{eqnarray*}
It follows that
\[\limsup_{t\downarrow0}\mathbb{P}\{|\delta(V,W^*_t)-\delta(V,\id)|>\varepsilon\}=0\quad\text{for all }\varepsilon>0,\quad\text{for all }V\in\wiren,\quad\text{for all }n\in\Nb.\]
Now, for any $\varepsilon>0$, the event $\{\xnorm{|W_t^*|-\mathbf{I}}>\varepsilon\}$ implies
\[\xnorm{|W_t^*|-\mathbf{I}}=\sum_{n\in\Nb}2^{-n}\sum_{V\in\wiren}||W_t^*|(V)-\mathbf{I}(V)|>\varepsilon.\]
Since $\sum_{V\in\wiren}||W_t^*|(V)-\mathbf{I}(V)|\leq 2$ for every $n\geq1$, we observe that
\[\sum_{n\geq \lceil1-\log_2(\varepsilon)\rceil}2^{-n}\sum_{V\in\wiren}||W_t^*|(V)-\mathbf{I}(V)|\leq 2\sum_{n\geq \lceil1-\log_2(\varepsilon)\rceil}2^{-n}=2^{\lfloor\log_2(\varepsilon)\rfloor}\leq\varepsilon.\]
Writing $m_{\varepsilon}=\lceil 2-\log_2(\varepsilon)\rceil$, we must have
\[|\delta(V,W_t^*)-\delta(V,\id)|>(\varepsilon-2^{-m_{\varepsilon}+1})4^{-{m_{\varepsilon}\choose{2}}}>0\]
for some $V\in\bigcup_{n<m_{\varepsilon}}\wiren$.
With $\varrho:=(\varepsilon-2^{-m_{\varepsilon}+1})4^{-{m_{\varepsilon}\choose{2}}}>0$, we now conclude that
\begin{eqnarray*}
\limsup_{t\downarrow0}\mathbb{P}\{\xnorm{|W_t^*|-\mathbf{I}}>\varepsilon\}&\leq&\limsup_{t\downarrow0}\mathbb{P}\left\{\bigcup_{n\leq m_{\varepsilon}}\bigcup_{V\in\wiren}\{|\delta(V,W_t^*)-\delta(V,\id)|>\varrho\}\right\}\\
&\leq&\limsup_{t\downarrow0}\sum_{n\leq m_{\varepsilon}}\sum_{V\in\wiren}\mathbb{P}\{|\delta(V,W_t^*)-\delta(V,\id))|>\varrho\}\\
&\leq&\sum_{n\leq m_{\varepsilon}}\sum_{V\in\wiren}\limsup_{t\downarrow0}\mathbb{P}\{|\delta(V,W_t^*)-\delta(V,\id)|>\varrho\}\\
&=&0.
\end{eqnarray*}
Where the interchange of sum and $\limsup$ is permitted because the sum is finite.
We conclude that $|W_t^*|\rightarrow_P\mathbf{I}$ as $t\downarrow0$ and $|W_t^*(\Gamma_0)|=|\Gamma_0||W_t^*|\rightarrow_P|\Gamma_0|$ as $t\downarrow0$.
The Feller property follows.

\end{proof}

\begin{proof}[Discontinuities]

Let $\Gbf_D=(\Gamma_t)_{t\geq0}$ have exchangeable initial distribution $\gamma_D$.
Then $(|\Gamma_t|)_{t\geq0}$ exists almost surely and has the Feller property.
By the Feller property, $|\Gbf_D|$ has a version with c\`adl\`ag paths.
For the c\`adl\`ag version, $\lim_{s\uparrow t}|\Gamma_t|$ exists for all $t>0$ with probability one.
%Intuitively, the discontinuities of $|\Gbf_D|$ must occur at discontinuity times of $\Gbf_D$, of which there are at most countably many.
Although the map $|\cdot|:\graphsN\rightarrow\limitdensities$ is not continuous, we also have $\lim_{s\uparrow t}|\Gamma_s|=|\Gamma_{t-}|$, as we now show.

Exchangeability of the initial distribution $\gamma_D$ as well as the transition probability implies $\Gbf_D\equalinlaw\Gbf_D^{\sigma}$ for all $\sigma\in\symmetricN$; therefore, $\Gamma_{t-}$ is exchangeable for all $t>0$.
By the Aldous--Hoover theorem (Theorem \ref{thm:A-H}), $|\Gamma_{t-}|$ exists a.s.\ for all $t>0$.

Now, suppose $t>0$ is a discontinuity time of $\Gbf_D$.
If $\lim_{s\uparrow t}|\Gamma_s|\neq|\Gamma_{t-}|$, then there exists $\varepsilon>0$ such that
\[\xnorm{\lim_{s\uparrow t}|\Gamma_s|-|\Gamma_{t-}|}>\varepsilon.\]
In particular, there exists $m\in\Nb$ and $F\in\graphsm$ such that 
\[|\lim_{s\uparrow t}\delta(F,\Gamma_s)-\delta(F,\Gamma_{t-})|>\varepsilon.\]
By definition,
\[\delta(F,\Gamma_s)=\lim_{n\rightarrow\infty}\frac{1}{n^{\downarrow m}}\sum_{\phi:[m]\rightarrow[n]}\mathbf{1}\{\Gamma_s^{\phi}=F\}.\]
Since the above limits exist with probability one, we can replace limits with limits inferior to obtain
\begin{eqnarray*}
\lefteqn{0\leq|\liminf_{s\uparrow t}\liminf_{n\rightarrow\infty}\frac{1}{n^{\downarrow m}}\sum_{\phi:[m]\rightarrow[n]}\mathbf{1}\{\Gamma_s^{\phi}=F\}-\mathbf{1}\{\Gamma_{t-}^{\phi}=F\}|\leq}\\
&\leq&
\liminf_{s\uparrow t}\liminf_{n\rightarrow\infty}\frac{1}{n^{\downarrow m}}\sum_{\phi:[m]\rightarrow[n]}|\mathbf{1}\{\Gamma_s^{\phi}=F\}-\mathbf{1}\{\Gamma_{t-}^{\phi}=F\}|\\
&\leq&2\liminf_{s\uparrow t}\liminf_{n\rightarrow\infty}\frac{1}{n^{\downarrow m}}\sum_{\phi:[m]\rightarrow[n]}\mathbf{1}\{\Gbf_D^{\phi}\text{ discontinuous on }[s,t)\}.
\end{eqnarray*}
By the bounded convergence theorem, Fatou's lemma, exchangeability and consistency of $\Gbf$, and the construction of $\Gbf$ from a standard $\omega$-process, we have
\begin{eqnarray*}
\lefteqn{\mathbb{E}\left[\liminf_{s\uparrow t}\liminf_{n\rightarrow\infty}\frac{1}{n^{\downarrow m}}\sum_{\phi:[m]\rightarrow[n]}\mathbf{1}\{\Gbf_D^{\phi}\text{ discontinuous on }[s,t)\}\right]\leq}\\
&\leq&\liminf_{s\uparrow t}\liminf_{n\rightarrow\infty}\frac{1}{n^{\downarrow m}}\sum_{\phi:[m]\rightarrow[n]}\mathbb{E}\left[\mathbf{1}\{\Gbf_D^{\phi}\text{ discontinuous on }[s,t)\}\right]\\
&\leq&\liminf_{s\uparrow t}\liminf_{n\rightarrow\infty}1-\exp\{-(t-s)\omega(\{W\in\wireN:\,W|_{[m]}\neq\mathbf{Id}_{[m]}\})\}\\
&=&0.
\end{eqnarray*}
By Markov's inequality, we must have
\[\mathbb{P}\left\{|\lim_{s\uparrow t}\delta(F,\Gamma_s)-\delta(F,\Gamma_{t-})|>\varepsilon\right\}=0\quad\text{for all }\varepsilon>0\text{ and all }F\in\mathcal{G}^*.\]
It follows that 
\[\mathbb{P}\{\xnorm{\lim_{s\uparrow t}|\Gamma_s|-|\Gamma_{t-}|}>\varepsilon\}=0\quad\text{for all }\varepsilon>0.\]

To see that the discontinuities in $|\mathbf{\Gamma}|$ occur only at the times of Type (III) discontinuities, suppose $s>0$ is a discontinuity time for $|\mathbf{\Gamma}_D|$.
The c\`adl\`ag property of $|\Gbf_D|$ along with the fact that $\lim_{s\uparrow t}|\Gamma_s|=|\lim_{s\uparrow t}\Gamma_s|$ a.s.\ implies that $|\delta(F,\Gamma_{s-})-\delta(F,\Gamma_s)|>\varepsilon$ for some $F\in\bigcup_{m\in\Nb}\mathcal{G}_{[m]}$ and some $\varepsilon>0$.
By the strong law of large numbers,
\[\lim_{n\rightarrow\infty}\frac{2}{n(n-1)}\sum_{1\leq i<j\leq n}\mathbf{1}\{\Gamma^{ij}_{s-}\neq\Gamma^{ij}_s\}>0,\]
because otherwise exchangeability would imply $\delta(F,\Gamma_{s-})=\delta(F,\Gamma_s)$;
therefore, there must by infinitely many edges with discontinuity at time $s$ with probability one.
Type (II) discontinuities can involve possibly infinitely many edges; however, there is a single fixed vertex $i^*$ for which all edges not involving $i^*$ remain constant.  So, if $s$ were the time of a Type (II) discontinuity, then 
\[\lim_{n\rightarrow\infty}\frac{2}{n(n-1)}\sum_{1\leq i < j\leq n}\mathbf{1}\{\Gamma^{ij}_{s-}\neq \Gamma^{ij}_s\}\leq\lim_{n\rightarrow\infty}\frac{2}{n(n-1)}\cdot n=0,\]
a contradiction.
Likewise, Type (I) discontinuities only involve a single edge, which cannot affect the limiting density of any subgraph.
The proof is complete.

\end{proof}

\end{proof}

\section{Concluding remarks}\label{section:concluding remarks}

Our main theorems characterize the behavior of exchangeable Feller processes on the space of countable undirected graphs.
These processes appeal to many modern applications in network analysis.
Our arguments rely mostly on the Aldous--Hoover theorem (Theorem \ref{thm:A-H}) and its extension to conditionally exchangeable random graphs (Theorem \ref{thm:discrete-rep}) and, therefore, our main theorems can be stated more generally for exchangeable Feller processes on countable $d$-dimensional arrays taking values in any finite space.
 In particular, our main theorems have an immediate analog to processes in the space of directed graphs and multigraphs.

In some instances, it may be natural to consider a random process on undirected graphs as the projection of a random process on directed graphs.
Any $G\in\graphsN$ projects to an undirected graph in at least two non-trivial ways, which we denote $G_{\vee}$ and $G_{\wedge}$ and define by
\begin{align*}
G_{\vee}^{ij}&:=G^{ij}\vee G^{ji}\quad\text{ and}\\
G_{\wedge}^{ij}&:= G^{ij}\wedge G^{ji}.
\end{align*}
Therefore, there is an edge between $i$ and $j$ in $G_{\vee}$ if there is any edge between $i$ and $j$ in $G$, while there is an edge between $i$ and $j$ in $G_{\wedge}$ only if there are edges $i$ to $j$ and $j$ to $i$ in $G$.  
It is reasonable to ask when an exchangeable Feller process $\mathbf{\Gamma}$ on directed graphs projects to an exchangeable Feller process on the subspace of undirected graphs.
But these questions are easily answered by consultation with Theorems \ref{thm:discrete char} and \ref{thm:Levy-Ito} and their generalization to directed graphs.
Similar questions might arise for exchangeable processes on similarly structured, but possibly higher-dimensional, spaces.
In these cases, the analogs to Theorems \ref{thm:discrete char}-\ref{thm:induced density process} can be deduced, so we omit them.

\bibliography{crane-refs}

\begin{thebibliography}{10}

\bibitem{Aldous1981a}
D.~Aldous.
\newblock Representations for {P}artially {E}xchangeable {A}rrays of {R}andom
  {V}ariables.
\newblock {\em Journal of Multivariate Analysis}, 11:581--598, 1981.

\bibitem{AldousExchangeability}
D.~J. Aldous.
\newblock Exchangeability and related topics.
\newblock In {\em \'{E}cole d'\'et\'e de probabilit\'es de {S}aint-{F}lour,
  {XIII}---1983}, volume 1117 of {\em Lecture Notes in Math.}, pages 1--198.
  Springer, Berlin, 1985.

\bibitem{BertoinLevy}
J.~Bertoin.
\newblock {\em L\'evy {P}rocesses}.
\newblock Cambridge University Press, Cambridge, 1996.

\bibitem{Bertoin2006}
J.~Bertoin.
\newblock {\em Random fragmentation and coagulation processes}, volume 102 of
  {\em Cambridge Studies in Advanced Mathematics}.
\newblock Cambridge University Press, Cambridge, 2006.

\bibitem{Hoover1979}
D.~Hoover.
\newblock Relations on {P}robability {S}paces and {A}rrays of {R}andom
  {V}ariables.
\newblock {\em Preprint, Institute for Advanced Study, Princeton, NJ}, 1979.

\bibitem{KallenbergRM}
O.~Kallenberg.
\newblock {\em Random Measures}.
\newblock Academic Press, New York, 1986.

\bibitem{KallenbergSymmetries}
O.~Kallenberg.
\newblock {\em Probabilistic {S}ymmetries and {I}nvariance {P}rinciples}.
\newblock Probability and Its Applications. Springer, 2005.

\bibitem{LovaszBook}
L.~Lov\'asz.
\newblock {\em Large Networks and Graph Limits}, volume~60 of {\em AMS
  Colloquium Publications}.
\newblock American Mathematical Society, Providence, RI, 2012.

\bibitem{LovaszSzegedy2006}
L.~Lov\'asz and B.~Szegedy.
\newblock Limits of dense graph sequences.
\newblock {\em J. Comb. Th. B}, 96:933--957, 2006.

\bibitem{Pitman2005}
J.~Pitman.
\newblock {\em Combinatorial stochastic processes}, volume 1875 of {\em Lecture
  Notes in Mathematics}.
\newblock Springer-Verlag, Berlin, 2006.
\newblock Lectures from the 32nd Summer School on Probability Theory held in
  Saint-Flour, July 7--24, 2002, With a foreword by Jean Picard.

\end{thebibliography}
\bibliographystyle{abbrv}

\end{document}